\documentclass[12pt]{amsart}
\usepackage{amsmath,amssymb,amsfonts,latexsym}
\usepackage{mathrsfs}
\usepackage[dvips]{geometry}
\usepackage{array}
\usepackage{enumitem}

\usepackage{amsfonts}
\usepackage[all,cmtip]{xy}
\include{diagxy}
\usepackage{ytableau}

\newtheorem*{lemma}{Lemma}

\newtheorem*{thm}{Theorem}
\newtheorem*{cor}{Corollary}
\newcommand{\iso}{\overset{\sim}{\rightarrow}}

\newcommand{\twoheaddownarrow}{\overset{\sim}{\twoheaddownarrow}}

\newcommand{\nc}{\newcommand}
\nc{\Ker}{\operatorname{Ker}} \nc{\rker}{\operatorname{rKer}}
\nc{\im}{\operatorname{Im}}
\nc{\stab}{\operatorname {Stab}}
\nc{\ann}{\operatorname {Ann}}
\nc{\Id}{\operatorname {Id}}
\nc{\Prim}{\operatorname {Prim}}
\nc{\Real}{\operatorname {Re}}
\nc{\Ext}{\operatorname {Ext}}
\nc{\rad}{\operatorname {rad}}
\nc{\sn}{\operatorname {sn}}
\nc{\wt}{\operatorname {wt}}
\nc{\Max}{\operatorname {Max}}
\nc{\height}{\operatorname {ht}}
\nc{\Supp}{\operatorname {Supp}}
\begin{document}

\title[Catalan sets]{A new interpretation of the Catalan numbers}
\thanks{$^1$ Work supported in part by the Israel Scientific Foundation grant 797/14.}
\author[Anthony Joseph and Polyxeni Lamprou]{Anthony Joseph and Polyxeni Lamprou}
{}

\maketitle


\begin{abstract}
Towards the study of the Kashiwara $B(\infty)$ crystal, sets $H^t,\, t\in \mathbb N$ of functions were introduced given by equivalence classes of unordered partitions satisfying certain boundary conditions \cite{J}.

Here it is shown that $H^t$ is a Catalan set of order $t$, that is to say the cardinality of $H^t$ is the $t${th} Catalan number $\mathcal C_t$.  This is a new description of a Catalan set and moreover admits some remarkable features.

Thus to $H^t$ there is an associated labelled graph $\mathscr G_t$ which is shown to have a \textit{canonical} decomposition into $(t-1)!$ subgraphs each with $2^{t-1}$ vertices.  These subgraphs, called $S$-graphs, have some tight properties which are needed for the study of $B(\infty)$.  They are described as labelled hypercubes in $\mathbb R^{t-1}$ whose edges connecting vertices with equal labels are missing.

It is shown that the number of \textit{distinct} hypercubes so obtained is again a Catalan number, namely $\mathcal C_{t-1}$. They define functions which depend on a coefficient set $(c_1,c_2,\ldots,c_{t-1})$ of non-negative integers. When the latter are non-zero and pairwise distinct, the vertices of the $S$-graphs describe distinct functions.  Moreover this property is retained if certain edges are deleted and certain vertices identified.  In particular when these coefficients are all equal and non-zero, it is shown that every hypercube degenerates to a simplex, resulting in exactly $t$ distinct functions, which for example are exactly those needed in the description of $B(\infty)$ in type $A$.

Some examples of $S$-graphs not of the type obtained above are also given.
\end{abstract}
\medskip


\section{Introduction}\label{1}
Throughout this paper, we denote by $\mathbb N=\{0, 1, 2, \ldots\}$ the set of natural numbers.

\subsection{} \label{1.1}
The Catalan numbers form a sequence of integers; for all $t\in \mathbb N$, the $t$th Catalan number is given by the formula $$\mathcal C_t=\frac{(2t)!}{t!(t+1)!}$$ or, recursively by $$\mathcal C_t=\sum\limits_{j=0}^{t-1}\mathcal C_{t-1-j}\mathcal C_j,$$ with $\mathcal C_0=1$. A collection of sets $X_t,\, t\in \mathbb N$ such that $|X_t|=\mathcal C_t$ for all $t\in \mathbb N$ is called a Catalan set. Many examples of Catalan sets are known; the triangulations of the $(t+2)$-gon, the Dyck paths from $(0,0)$ to $(0, 2t)$ and the nilpotent ideals in the Borel subalgebra of $\mathfrak{sl}_t$.  In \cite{S} one can find 66 examples of Catalan sets and a few more in \cite{S2}.

\subsection{} \label{1.2}
In this paper we give a new example of a Catalan set $H^t$ which appeared in \cite{J} while studying the Kashiwara crystal $B(\infty)$. The elements of $H^t$ are equivalence classes of unordered partitions (Section \ref{2.1}) into $t+1$ with boundary conditions (Section \ref{2.2}). They parametrize certain functions which are involved in an inductive construction of ``dual Kashiwara functions'', which in turn are used to describe $B(\infty)$. The observation that the cardinalities of $H^t$ for small $t$ are given by the Catalan numbers is due to S. Zelikson who suggested that this should be true in general. We prove the Zelikson conjecture in Theorem \ref{4.1}.

\subsection{} \label{1.3}
Most interestingly, our Catalan set $H^t$ has the following remarkable property: it decomposes into a union of $(t-1)!$ subsets, each of cardinality $2^{t-1}$ and whose common union has $t$ elements. Moreover, this decomposition is canonical if one considers the {\it graph of links} $\mathscr G_t$ of $H^t$, introduced in \cite[Section 6]{J}. This graph has labels on both vertices and edges. These follow a number of rules (Section \ref{5}). Then the aforementioned subsets of cardinality $2^{t-1}$ are the vertices of subgraphs ($S$-graphs) with some very special properties vital to the construction of dual Kashiwara functions. It is shown that $\mathscr G_t$ has a unique decomposition into $S$-graphs each determined by the natural order relation on a set of $t-1$ non-negative pairwise distinct integers. Curiously, there are exactly $\mathcal C_{t-1}$ isomorphism classes of such labelled graphs, rather than the expected $(t-1)!$. These $S$-subgraphs of $\mathscr G_t$ are labelled hypercubes whose edges connecting vertices having equal labels are missing.

\subsection{} \label{1.4}
In the example of nilpotent ideals in the Borel subalgebra of $\mathfrak{sl}_t$, two subsets of cardinality $2^{t-1}$ are known, namely the commutative ideals (a result due to Peterson published in \cite{K}) and the nilradicals of the parabolic subalgebras of $\mathfrak{sl}_t$. However, even in the case $t=3$ (and so $\mathcal C_2=2$) the decomposition of the set of nilpotent ideals of the Borel subalgebra of $\mathfrak{sl}_3$ into the two subsets of cardinality $4$ is quite different from the decomposition of $H^3$. Moreover, the graph $\mathscr G_t$ of links of $H^t$, is quite different to the graph of orderings of nilpotent ideals (Section \ref{7.4}).

\subsection{} \label{1.5}
The vertices of $\mathscr G_t$ with labels $t$ are called extremal. Via Theorem \ref{4.1}, their cardinality is again the Catalan number $\mathcal C_{t-1}$. It is shown in \ref{5.7} that each $S$-subgraph is, after deleting some edges, a disjoint union of trees rooted at the extremal elements. This is of some importance in constructing the dual Kashiwara functions for $B(\infty)$.

\subsection{} \label{1.6}
The vertices of an $S$-subgraph of $\mathscr G_t$ define functions depending on $t-1$ positive integers $c_1, c_2,\dots c_{t-1}$. These functions are distinct when the $c_i,\, 1\le i\le t-1$ are pairwise distinct, but start to coincide when the $c_i$ coalesce. The corresponding hypercubes can be degenerated so that vertices still describe distinct functions. In the most extreme case when the $c_i$ are all equal, every hypercube degenerates to a simplex (Section \ref{5.8}).\\


\noindent {\bf Acknowledgement.} We would like to thank S. Zelikson for his valuable contribution to this project that started in \cite{J} and is still on-going. In particular, we would like to thank him for pointing out that the cardinalities of the sets $H^t$ are the Catalan numbers as well as for bringing \cite{BKR} to our attention.

\section{Diagrams of order $t+1$ : definition and properties}\label{2}
In this section we recall some definitions and results given in \cite{J} that will be needed in the sequel.

\subsection{Unordered Partitions}\label{2.1}

Let $t$ be a non-negative integer.  A diagram $\mathscr D$ of order $t+1$ is a collection of $t+1$ columns $C_1,C_2,\ldots,C_{t+1}$ placed on the $x$-axis with the subscript denoting the $x$ co-ordinate of the corresponding column. We view the columns as a stack of square blocks of equal size with the lowest block of each column (if it is not empty) having the same $y$ co-ordinate.  The rows $R_i:i=1,2,\ldots,$ are indexed by their $y$ co-ordinate. The height of $C_i$, denoted by $\operatorname{ht}\, C_i$, is defined to be the number of blocks it contains and the height of a diagram $\mathscr D$, denoted by $\operatorname{ht}\, \mathscr D$, is defined to be the maximal height of its columns. Then $\mathscr D$ is just the presentation of the unordered partition $(\operatorname{ht}\,C_1,\operatorname{ht}\,C_2,\ldots,\operatorname{ht}\,C_{t+1})$.

\subsection{Boundary Conditions}\label{2.2}
Let $s$ be a positive integer. Distinct columns $C,C'$ of a diagram $\mathscr D$ are said to be neighbouring at height $s$ if they have height $\geq s$ and every column between them has height $<s$.

A left (resp. right) extremal column $C$ of $\mathscr D$ is one which has no neighbours to the left (resp. right) at height $\operatorname{ht}\, C$.

A column $C$ of a diagram $\mathscr D$ is said to be strongly extremal if it is of the maximal height $r:=\operatorname{ht}\, \mathscr D$  and  is left (resp. right) extremal if $r$ is odd (resp. even). This can also be expressed as saying that $C$ is the leftmost (resp. rightmost) highest column of $\mathscr D$. As a consequence $\mathscr D$ admits exactly one strongly extremal column. For the empty diagram of order $t+1$ we take by convention $C_{t+1}$ to be its strongly extremal column (of height zero).\\

\noindent {\bf Definition.} The left (resp. right) boundary condition on $\mathscr D$ is that the height of any left (resp. right) extremal column of $\mathscr D$ must be either even (resp. odd) or the height of $\mathscr D$.\\

From now on, we only consider diagrams which satisfy both the left and right boundary conditions. Notice that this means that $C_{t+1}$ cannot be empty, unless the diagram is empty.\\

\noindent {\bf Example.}\\
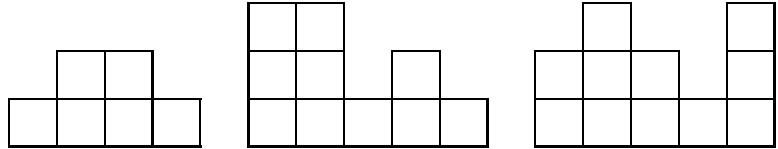
\[\ytableausetup{notabloids}
\begin{ytableau}
\none & \none & \none & \none & \none & *(white) & *(white) & \none & \none & \none &\none & \none & *(white) & \none & \none &  *(white)\\
\none & *(white) & *(white) & \none & \none & *(white) & *(white) & \none & *(white) & \none & \none & *(white) & *(white) & *(white) &  \none  & *(white)\\
*(white) & *(white) & *(white) & *(white) & \none & *(white) & *(white) &  *(white) & *(white)  & *(white) & \none & *(white)  & *(white)  & *(white)  & *(white) &  *(white)
\end{ytableau}\]

\noindent Only the third diagram satisfies the boundary conditions.

\subsection{Operations on diagrams}\label{2.3}
We recall some constructions and results given in \cite[Sections 2, 3]{J}.

\subsubsection{Adjunction/removal of dominoes}\label{2.3.1}

A domino $D$ is a pair of blocks lying on the topmost and second to topmost rows of the same column.

Consider a domino $D$ being adjoined to a column $C$ of height $i$ in a diagram $\mathscr D$ to give a new column $C \sqcup D$ of height $i+2$.  Then $D$ is called an even (resp. odd) domino if $i$ is even (resp. odd) and is called a left (resp. right) domino if $C \sqcup D$ is the left (resp. right) neighbour at heights $i+1,i+2$ of a column $C'$ of height $\geq i+2$ in the new diagram.

The diagram $\mathscr D \sqcup D$ obtained by adjoining a left even or right odd domino $D$ to a diagram $\mathscr D$ again satisfies the boundary conditions. Moreover, this operation does not change the height of the diagram $\mathscr D$, nor its strongly extremal column. We call the map $\mathscr D \rightarrow \mathscr D \sqcup D$ domino adjunction. The inverse of this map is called domino removal.

A diagram to which no dominoes can be adjoined is called {\it complete}. Starting with a diagram $\mathscr D$ we may adjoin to it (necessarily finitely many) dominoes until we obtain a complete diagram $\widehat{\mathscr D}$. We call the operation $\mathscr D\rightarrow \widehat{\mathscr D}$ the completion of $\mathscr D$. Note that $\mathscr D$ and $\widehat{\mathscr D}$ have the same height.

A diagram from which no dominoes can be removed is called {\it deplete}.\\

\noindent {\bf Example.}
We may complete the diagram on the left (which is deplete) by adding a left even domino to the third column:

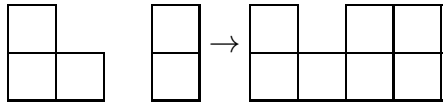
\[\ytableausetup{notabloids}
\begin{ytableau}
*(white) & \none & \none & *(white)\\
*(white) & *(white) & \none & *(white)
\end{ytableau}
\rightarrow
\ytableausetup{notabloids}
\begin{ytableau}
*(white) & \none & *(white) & *(white)\\
*(white) & *(white) & *(white) & *(white)
\end{ytableau}\]

\subsubsection{Adjunction/removal of half-dominoes}\label{2.3.2}
Let $C$ be the strongly extremal column of $\mathscr D$. Consider a single block adjoined to the top of $C$. The new diagram $\mathscr D^+$ satisfies the boundary conditions and the column $C^+$ ($C$ with a block on top) is its strongly extremal column. We call the operation $\mathscr D\rightarrow \mathscr D^+$ half-domino adjunction. Note that even when $\mathscr D$ is complete $\mathscr D^+$ might not be complete.

Similarly, if a diagram $\mathscr D$ has a unique column $C$ of maximal height (which is necessarily its strongly extremal column), we may remove the top block to obtain a new diagram $\mathscr D^-$ which satisfies the boundary conditions, since $\mathscr D$ does. Then the column $C^-$ (defined to be $C$ after removal of its top block) is the strongly extremal column of $\mathscr D^-$.\\

\noindent {\bf Example.} The strongly extremal column of the diagram below is the last one; we may add a block on top of it:
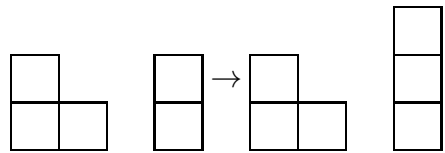
\[\ytableausetup{notabloids}
\begin{ytableau}
\none & \none  & \none & \none & \none &\none  & \none & \none & *(white) \\
*(white) & \none & \none & *(white) &  \none[\rightarrow] & *(white) & \none & \none & *(white)\\
*(white) & *(white) & \none & *(white)& \none & *(white) & *(white) & \none & *(white)
\end{ytableau}\]

\noindent \textbf{Remark.} (Repeated half-domino adjunction). We may repeat the process of half-domino adjunction to $\widehat{\mathscr D^+}$.  In this $C^+$ is again the strongly extremal column in $\widehat{\mathscr D^+}$ which is now of height $\height\, \mathscr D+1$.  The new diagram can be obtained by just adding two complete rows to the bottom of the diagram or by adjoining a vertical domino to the strongly extremal column and completing.

\subsubsection{Adjunction/removal of equal adjacent rows}\label{2.3.3}
If two adjacent rows $R_{i+1},R_{i+2}$ admit the same number of boxes, then they may be cancelled to obtain a diagram of height decreased by $2$ without upsetting the boundary conditions. Vice versa, two equal rows may be added to a diagram, to obtain a new one, with height increased by $2$ which still satisfies the boundary conditions. Note that adjuction/removal of two identical rows does not change the strongly extremal column of a diagram. A diagram in which no two adjacent rows can be cancelled is called {\it reduced}.\\

\noindent {\bf Example.} We may add two complete rows at the bottom of the diagram on the left:
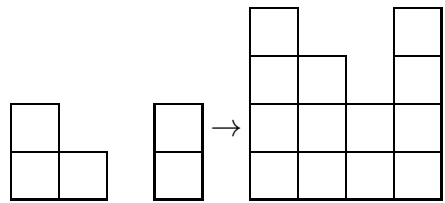
\[\ytableausetup{notabloids}
\begin{ytableau}
\none & \none  & \none & \none & \none & *(white) & \none & \none & *(white)\\
\none & \none  & \none & \none &\none & *(white) & *(white) & \none & *(white)\\
*(white) & \none & \none & *(white) & \none[\rightarrow] & *(white) & *(white) & *(white) & *(white)\\
*(white) & *(white) & \none & *(white) & \none & *(white) & *(white) & *(white) & *(white)
\end{ytableau}\]

\subsubsection{}\label{2.3.4}
The above three operations, namely, adjunction (and removal) of dominoes (and half-dominoes) and adjunction (and removal) of equal adjacent rows may be combined to give an equivalence relation on diagrams. We denote by $[\mathscr D]$ the equivalence class of a diagram $\mathscr D$.

The set of equivalence classes of diagrams of order $t+1$ is denoted
by $H^{t+1}$.  For any representative in $[\mathscr D]$ the $x$ co-ordinate of a strongly extremal column is independent of the choice of representative, hence we may talk about the strongly extremal column of $[\mathscr D]$.  Then for
any $j$, with $1\le j\le t+1$ we let $H_j^{t+1}$ denote the set of equivalence classes in $H^{t+1}$ whose strongly extremal column is $C_j$.

\subsection{Properties}\label{2.4}
\begin{enumerate}
\item \cite[Lemma 2.3.6]{J} A complete diagram is reduced if and only if its two bottom rows are not identical (equivalently, if it has a column of height $\le 1$).
\item \cite[Lemma 2.3.7]{J} Every equivalence class of diagrams admits exactly one deplete diagram $\mathscr D$. Let $h$ be its height; then, for every $k\in \mathbb N$ there exists exactly one complete diagram of height $h+k$ in the equivalence class of $\mathscr D$. Two of these are reduced. Their heights are $h$ and $h+1$. 
\item Let $[\mathscr D]\in H_j^{t+1}$, where $\mathscr D$ is a complete reduced diagram. Then $\widehat{\mathscr D^+}\in [\mathscr D]$ is complete, $C_j$ is its strongly extremal column and it obtains from $\mathscr D$ by adjoining two bottom rows to all columns to the right (resp. left) of $C_j$ if $\operatorname{ht}\, \mathscr D$ is even (resp. odd) \cite[Lemma 2.3.4]{J}. Hence, by $(1)$ and $(2)$ above, $\operatorname{ht}\,\widehat{\mathscr D^+}=\operatorname{ht}\, \mathscr D-1$ if and only if all columns to the left (resp. right) of $C_j$ have height $\ge 2$. Otherwise, $\operatorname{ht}\,\widehat{\mathscr D^+}=\operatorname{ht}\, \mathscr D+1$.
\item \cite[Lemma 2.3.2]{J} Let $\mathscr D$ be a complete diagram of height $r$. Then
\begin{enumerate}
\item $\operatorname{ht}\,C_1 \geq r-1$ and $\operatorname{ht}\,C_1=r$ if $r$ is even.
\item $\operatorname{ht}\, C_{t+1}\geq r-1$ and $\operatorname{ht}\,C_{t+1}=r$ if $r$ is odd.
\end{enumerate}
\end{enumerate}

\subsection{Duality}\label{2.5}
For any diagram $\mathscr D$, we define its dual $\mathscr D^*$ to be the diagram that obtains from $\mathscr D$ by adding a complete bottom row and reversing the order of its columns. As an unordered partition, if $\mathscr D=(\height\,C_1, \height\,C_2, \dots, \height\,C_{t+1})$, then $\mathscr D^*=(\height\, C_{t+1}+1, \height\, C_t+1, \dots, \height\, C_1+1)$. The map $*:[\mathscr D]\rightarrow [\mathscr D^*]$ is an involution in $H^{t+1}$.\\

\noindent {\bf Example.} The dual diagram of
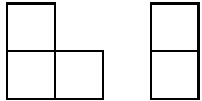
\[\ytableausetup{notabloids}
\begin{ytableau}
*(white) & \none & \none & *(white) \\
*(white) & *(white) & \none & *(white)
\end{ytableau}\]
is the diagram
\[\ytableausetup{notabloids}
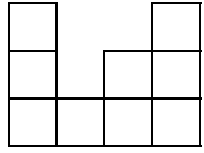
\begin{ytableau}
 *(white) & \none & \none & *(white)\\
*(white) & \none & *(white) & *(white) \\
*(white) & *(white) & *(white) & *(white)
\end{ytableau}\]
Then the dual of the latter diagram is
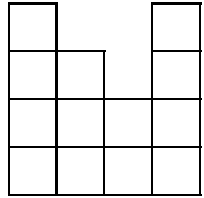
\[\ytableausetup{notabloids}
\begin{ytableau}
*(white) & \none & \none & *(white)\\
*(white) & *(white) & \none & *(white)\\
*(white) & *(white) & *(white) & *(white)\\
*(white) & *(white) & *(white) & *(white)
\end{ytableau}\]
which is equivalent to the initial diagram by \ref{2.3.3}.

\section {Catalan sets}\label{3}

\subsection{}\label{3.1}
Recall that $H^{t+1}$ denotes the set of equivalence classes of diagrams of order $t+1$. We computed by hand that $|H^{t+1}|$ equals $1, 2, 5, 14, 42$ for $t=0, 1, 2, 3, 4$.  (Note that $H^1$ consists of the empty diagram with a single empty column).

S. Zelikson noticed that the above numbers are the first Catalan numbers and suggested that in general $|H^{t+1}|=\mathcal C_{t+1}$, where we recall that $\mathcal C_t= \frac{(2t)!}{(t+1)!t!}$.

We prove the Zelikson conjecture in Section \ref{4.1}. In Sections \ref{3.2} and \ref{3.3} we give preliminary lemmata for this purpose. It turns out that a more general result holds. Indeed, we will show in \ref{4.1} that for all $j$, with $1\le j\le t+1$, one has that $|H^{t+1}_j|=\mathcal C_{j-1}\mathcal C_{t-j+1}$. In particular, $|H_{t+1}^{t+1}|=\mathcal C_t$.

\subsection{}\label{3.2}
Let $\mathscr D$ be a diagram and $[\mathscr D]$ its equivalence class. We define the height $\height\, [\mathscr D]$ of $[\mathscr D]$ to be the height of the unique deplete diagram in $[\mathscr D]$. It is the minimal height of a diagram in $[\mathscr D]$. Let $^rH^s$ denote the subset of elements of $H^s$ of height $r$. Since the number of blocks in each row of a reduced diagram strictly decreases from bottom to top, it follows that $^rH^s=\emptyset$ unless $0\leq r \leq s-1$.  Let $^rH_j^s$ (resp. $^{\leq r}H_j^s$) denote the subset of elements in $H_j^s$ of height $r$ (resp. at most $r$). Then, by the boundary conditions,

$$^rH^j_j=\emptyset,\quad \textrm{unless}\quad r \quad \textrm{is even and}\quad ^rH^j_1=\emptyset, \quad \textrm{unless}\quad r\quad \textrm{ is odd}.\eqno{(*)}$$

Suppose that $\mathscr D$ and $\mathscr D'$ are diagrams of order $j$ and $k$ respectively, such that the height of the last column $C_j$ of $\mathscr D$ is the same as the height of the first column $C_1'$ of $\mathscr D'$. Then we can merge them into a new diagram $\mathscr D''=(\height\, C_1, \height\, C_2, \dots, \height\, C_j=\height\, C'_1, \height\, C_2', \dots, \height\, C_k')$. It is immediate that $\mathscr D''$ satisfies the boundary conditions and thus is a diagram of order $j+k-1$. When $\height\, C_1' - \height\,C_j$ is a positive (resp. negative) even number $s$, we may still merge the two diagrams if we first adjust the height of $\mathscr D$ (resp. $\mathscr D'$) by adding to it $s$ (resp. $-s$) full rows.\\

\noindent {\bf Example.}
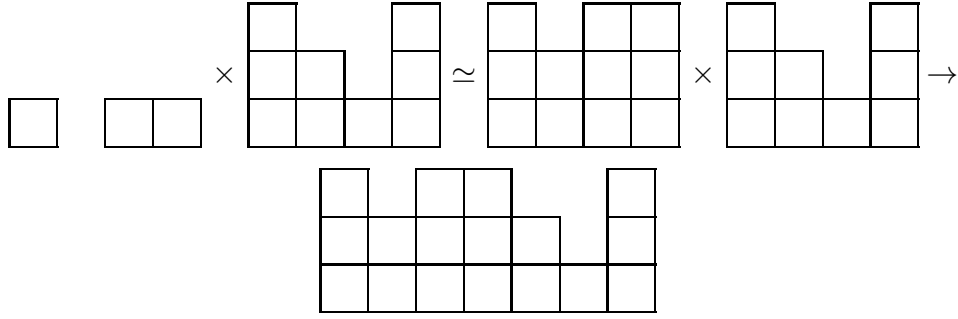
\[\ytableausetup{notabloids}
\begin{ytableau}
\none & \none  & \none & \none & \none & *(white) & \none & \none & *(white) & \none & *(white) & \none & *(white) & *(white)& \none & *(white) & \none & \none & *(white) & \none\\
\none & \none  & \none & \none & \none[\times] & *(white) & *(white) & \none & *(white)& \none[\simeq] & *(white) & *(white) & *(white) & *(white) & \none[\times] & *(white) & *(white) & \none & *(white) & \none[\rightarrow]\\
*(white) & \none & *(white) & *(white) & \none & *(white) & *(white) & *(white) & *(white) & \none & *(white) & *(white) & *(white) & *(white) & \none & *(white) & *(white) & *(white) & *(white) & \none
\end{ytableau}\]
\[\ytableausetup{notabloids}
\begin{ytableau}
*(white) & \none & *(white) & *(white) & \none & \none & *(white)\\
*(white) & *(white) & *(white) & *(white) & *(white) & \none & *(white)\\
*(white) & *(white) & *(white) & *(white) & *(white) & *(white) & *(white)
\end{ytableau}\]

Both initial diagrams are of order 4. We adjust the height of the first diagram by adding two complete rows so that its last column has height 3 equal to that of the first column of the second diagram. Then we merge them and obtain a diagram of order 7.

\begin {lemma} For all $t \in \mathbb
N$ and all $j$, with $1< j< t+1$, merging of height adjusted complete diagrams at the
$j^{th}$ column gives bijections
\begin{enumerate}
\item $^rH^j_j \times \ ^{\leq (r-1)}H_1^{t+2-j}  \iso \ ^rH^{t+1}_j$, if $r$ is even.
\item $^{\leq (r-1)}H^j_j \times \  ^rH_1^{t+2-j}\iso \ ^rH^{t+1}_j$, if $r$ is odd.
\end{enumerate}
\end {lemma}
\begin {proof}
We will prove the first statement. The proof of the second one is very similar and so we omit it.

Let $[\mathscr D_-] \in \ ^rH^j_j$ and $[\mathscr D_+] \in  \ ^{\leq (r-1)}H^{t+2-j}_1$, where $\mathscr D_-,\, \mathscr D_+$ are complete reduced diagrams of minimal height in their equivalence class.  Then the $j^{th}$ column of $\mathscr D_- $ has height $r$, whilst the first column of $\mathscr D_+$ has odd height $\leq (r-1)$. By shifting indices we regard the latter as its $j^{th}$ column. It is its strongly extremal column, hence we may add a half-domino to its top to make it of even height $\leq r$.  Moreover one may add an even number of full rows to the bottom of $\mathscr D_+$ to make it of height $r$.  Then since these two diagrams satisfy the boundary conditions they may be merged at their $j^{th}$ column of common height $r$ to give a diagram $\mathscr D$, which is complete since $\mathscr D_\pm$ are complete. Again $\mathscr D$ is reduced because $\mathscr D_-$ is reduced.  Finally $C_j$ is the rightmost column of height $r$. In order to show that $[\mathscr D]\in\ ^rH^{t+1}_j$, it remains to check that $\mathscr D$ is of minimal height in its equivalence class. If we add a box to $C_j$ to obtain $\mathscr D^+$, then its completion $\widehat{\mathscr D^+}$ obtains by exactly adding vertical dominoes on columns to the right of $C_j$.  Thus it is reduced because $\mathscr D_-$ is reduced and $\operatorname{ht}\,\widehat{\mathscr D^+}=r+1$.  Hence by \ref{2.4}$(3)$ $\mathscr D$ is of minimal height in its equivalence class.

The map $([\mathscr D_-], [\mathscr D_+])\rightarrow [\mathscr D]$ is clearly injective. We will show that it is also surjective, thus an isomorphism of sets $^rH^j_j \times \ ^{\leq (r-1)}H_1^{t+2-j}  \iso \ ^rH^{t+1}_j$.

Take $[\mathscr D] \in\ ^rH^{t+1}_j$, where $\mathscr D$ is the complete reduced diagram of height $r$. Recall that $r$ is even; by definition the rightmost column of maximal height $r$ of $\mathscr D$ is $C_j$. By \ref{2.4}$(4)$, $\operatorname{ht}\, C_1=r$. By \ref{2.4}$(3)$, at least one of the columns $C_1, C_2, \ldots, C_{j-1}$ has height $\le 1$.

The diagram $\mathscr D_-$ consisting of the first $j$ columns of $\mathscr D$ satisfies the left boundary conditions since $\mathscr D$ does and it also satisfies the right boundary conditions since $C_j$ is the rightmost column in
$\mathscr D_-$ and a fortiori $C_j$ is also a column of maximal
height in  $\mathscr D_-$. On the other hand, if one can adjoin a domino to $\mathscr D_-$, one can also adjoin the same domino to $\mathscr D$; the latter is impossible, since $\mathscr D$ is complete by assumption. Hence $\mathscr D_-$ is complete. By \ref{2.4}$(1)$, $\mathscr D_-$ is also reduced, since it has a column of height $\le 1$. The second complete reduced diagram in $[\mathscr D_-]$, namely $\widehat{\mathscr D_-^+}$, obtains from $\mathscr D_-$ by adding a box at its rightmost column $C_j$; hence $\operatorname{ht}\,\widehat{\mathscr D_-^+}=\operatorname{ht}\, \mathscr D_-+1$. We conclude that $[\mathscr D_-]\in\ ^rH_j^j$.

Let $\mathscr D_+$ be the diagram obtained from $\mathscr D$ by the removal of the first $j-1$ columns.  It satisfies the right boundary conditions since $\mathscr D$ does. It satisfies the left boundary condition since $C_j$ is the
leftmost column in $\mathscr D_+$ and a fortiori $C_j$ is also a
column of maximal height in  $\mathscr D_+$. As in the case of $\mathscr D_-$, it is complete. The column $C_j$ is the only column in $\mathscr D_+$ of maximal height $r$. By \ref{2.4}$(4)$, $\operatorname{ht}\, C_{t+1}=r-1$. Then we may remove a single block from $C_j$ to obtain a diagram $\mathscr D_+^-$ with at least two columns of height $r-1$, namely $C_j$ and $C_{t+1}$. This removal of a single block does not change the property of $\mathscr D_+$ being complete. If $\mathscr D_+^-$ is reduced, then it is the complete reduced diagram of minimal height $r-1$ in its equivalence class. If not, a complete reduced diagram of odd height $<r-1$ is obtained by cancelling an even number of bottom rows. It is of minimal height in its equivalence class, since its strongly extremal column is its first column.  We conclude that $[\mathscr D_+] \in \ ^{\leq (r-1)}H^{t+2-j}_1$.

Finally, observe that $\mathscr D$ is recovered by merging $\mathscr D_-$ and $\mathscr D_+$ at their
common column $C_j$. This completes the proof of surjectivity.
\end {proof}

\textbf{Remark.}  
The diagram  $\mathscr D_+$ in the penultimate paragraph will generally fail to be reduced.  For example take $\mathscr D=(4,3,2,1,4,2,3) \in \ ^4H^7_5$.  Then $\mathscr D_+ = (4,2,3)$ is complete but not reduced.

\subsection {}\label{3.3}
In the previous lemma we described $^rH^{t+1}_j$, for $1<j<t+1$. Below we will describe $^rH^{t+1}_{t+1}$ for $r$ even and $^rH^{t+1}_1$ for $r$ odd.
Here we identify $H^0,H^1,H^1_1$ with the (equivalence class of the) empty diagram.

\begin {lemma} Let $t\in \mathbb N$. Adjoining a single column (and possibly a single block) gives bijections
\begin{enumerate}
\item $^rH^t \sqcup \ ^{r-1}H^t \iso \ ^rH^{t+1}_1$, for $r$ odd.
\item $^rH^t \sqcup \ ^{r-1}H^t \iso \ ^{r}H^{t+1}_{t+1}$, for $r \geq 2$ and even.
\end{enumerate}
\end {lemma}
\begin {proof}  We will prove (1). The proof of (2) is very similar and so we omit it.

Take $[\mathscr D] \in \ ^rH^t$, where $\mathscr D$ is a complete reduced diagram of minimal height in its equivalence class.   By \ref {2.4}$(4)$ the first column of $\mathscr D$ has height $\geq r-1$.  Let $C\sqcup \mathscr D$ be the diagram of order $t+1$ obtained by adjoining a column of height $r$ on the left of $\mathscr D$.  It satisfies the right boundary condition since $\mathscr D$ does and the overall height has not been increased. It satisfies the left boundary condition since its first column has the maximal height $r$.  It is complete since $\mathscr D$ is complete and since its second column has height $\geq r-1$. Since its strongly extremal column is its first column $C$, $\widehat{(C\sqcup \mathscr D)^+}=(C\sqcup \mathscr D)^+$; hence $[C\sqcup\mathscr D] \in \ ^rH^{t+1}_1$, since $C\sqcup\mathscr D$ is of minimal height in its equivalence class.  The resulting map $^rH^t \rightarrow \ ^rH^{t+1}_1$ is clearly injective.

Let $[\mathscr D] \in \ ^{r-1}H^t$, where $\mathscr D$ is a complete reduced diagram of minimal height in its equivalence class.  Then $\widehat{\mathscr D^+}$, is also reduced and has height $r$. We proceed as before, namely we add a column of height $r$ on the left of $\widehat{\mathscr D^+}$; the resulting diagram $C\sqcup \widehat{\mathscr D^+}$ satisfies the boundary conditions and is complete and reduced. It is of height $r$ with strongly extremal column $C$, hence it is of minimal height in its equivalence class. Clearly the corresponding map $^{r-1}H^t \rightarrow \ ^rH^{t+1}_1$ is injective and has a disjoint image to the first map.

For surjectivity take $[\hat{\mathscr D}] \in \ ^rH^{t+1}_1$, where $\hat{\mathscr D}$ is complete reduced of minimal height in $[\hat{\mathscr D}]$.  Since $\hat{\mathscr D}$ is complete and its first column is of odd height $r$, it follows that the second column of $\hat{\mathscr D}$ has height $\geq r-1$. Then the diagram $\mathscr D$ defined by removing the first column of $\hat{\mathscr D}$ satisfies the boundary conditions and is complete.  It is reduced because $\hat{\mathscr D}$ is reduced. Its height is equal to $r$, since by \ref{2.4}$(4)$ $\operatorname{ht}\, C_{t+1}=r$. However it need not be of minimal height in its equivalence class.  This means that the height of $[\mathscr D]$ is either $r$ or $r-1$. This proves surjectivity.
\end {proof}

\noindent \textbf{Example.} Take $\hat{\mathscr D}=(3,2,1,3,2,3)$ which is complete reduced and of minimal height in $[\hat{\mathscr D}]$.  Then $\mathscr D =(2,1,3,2,3)$ is complete and reduced but not of minimal height in its equivalence class. Indeed, $(2,1,2,0,1)\in [\mathscr D]$ is complete and reduced of height $\operatorname{ht}\, \mathscr D-1$.

\section {Proof of the Zelikson conjecture}\label{4}
\subsection{}\label{4.1}
Let $^r \mathscr C^s_j$ denote the cardinality of $^rH^s_j$ and set $^r\mathscr C ^s =\sum\limits_{j=1}^{t+1} \ ^r\mathscr C ^s_j$.  We define the Catalan polynomial of order $t+1,\, t \in \mathbb N$ to be
$$\mathscr C _{t+1}(q):= \sum_{r=0}^\infty   q^r  \ ^r\mathscr C ^{t+1}.$$

Recall the notation and hypotheses of \ref {3.1}. Below we prove the Zelikson conjecture:

\begin {thm}  For all $t \in \mathbb N$ one has $|H^{t+1}|=\mathscr C _{t+1}(1)=\mathcal C_{t+1}$.  Moreover for all $j$ with $1\le j\le t+1$,
$$|H^{t+1}_j|= \mathcal C_{j-1}\mathcal C_{t-j+1}.$$
\end {thm}

\begin {proof}  One has $^0H^{t+1}= \ ^0H^{t+1}_{t+1}$; in particular $^0\mathscr C ^{t+1}= \ ^0\mathscr C ^{t+1}_{t+1}=1$
and  $^0\mathscr C ^{t+1}_j=0$, for all $j\leq t$. Note also that $|H^0|=1$ and $^0\mathscr C ^0=1$.

By Lemma \ref {3.3} we have that for all $k>s$
$$\  ^s\mathscr C ^k_1=\left\{\begin{array}{cc} \ ^{s-1}\mathscr C ^{k-1}+\ ^s\mathscr C ^{k-1}, &s \quad \textrm{odd},\\
0, & s \quad \textrm{even}.
\end{array}\right.\eqno{(1)}$$
and
$$\  ^s\mathscr C ^k_k=\left\{\begin{array}{cc} \ ^s\mathscr C ^{k-1}+\ ^{s-1}\mathscr C ^{k-1}, &s \quad \textrm{even},\\
0, & s \quad \textrm{odd}.
\end{array}\right.\eqno{(2)}$$

By Lemma \ref {3.2}, we have that for $r\ge 2$ and even and $2\le j\le t$

$$^r\mathscr C ^{t+1}_j=\sum_{s=0}^{r-1} \ ^r\mathscr C ^j_j \  ^s\mathscr C ^{t+2-j}_1,\eqno{(3)}$$
whereas for $r\ge 1$ and odd and $2\le j\le t$
$$^r\mathscr C ^{t+1}_j=\sum_{s=0}^{r-1} \ ^s\mathscr C ^j_j \  ^r\mathscr C ^{t+2-j}_1.\eqno{(4)}$$
Then by $(1)$ and $(2)$ equation $(3)$ becomes
$$^r\mathscr C ^{t+1}_j=\sum_{s=0}^{r-1} ( ^r\mathscr C ^{j-1}+ \ ^{r-1}\mathscr C ^{j-1}) \  ^s\mathscr C ^{t+1-j}, \eqno{(3')}$$
where $r\ge 2$ is even and $2\le j\le t$, and equation $(4)$ becomes
$$^r\mathscr C ^{t+1}_j=\sum_{s=0}^{r-1} \ ^s\mathscr C ^{j-1}( ^r\mathscr C ^{t+1-j}+ \ ^{r-1}\mathscr C ^{t+1-j}),\eqno{(4')}$$
where $r\ge 1$ is odd and $2\le j\le t$.
Note that since $^s\mathscr C ^0=0$, unless $s=0$, and $^0\mathscr C ^0=1$ equations $(3')$ and $(4')$ for $j=1,\, t+1$ coincide with $(1)$ and $(2)$,  unless $j=t+1$ and $r=1$, hence $(3')$ and $(4')$ are still valid for $j=1,\, t+1$, unless $j=t+1$ and $r=1$, in which case $(4')$ gives $1$, whereas $^1\mathscr C _{t+1}^{t+1}=0$.
The right hand side of $(4')$ is transformed into the right hand side of $(3')$ by the substitution $j\mapsto t+2-j$.
Thus summing over $j$ we deduce that for all $r\ge 1$
$$^r\mathscr C ^{t+1}=\left\{\begin{array}{cc} \sum\limits_{j=1}^{t+1} \ \sum\limits_{s=0}^{r-1} ( ^r\mathscr C ^{j-1}+ \ ^{r-1}\mathscr C ^{j-1}) \  ^s\mathscr C ^{t+1-j}, &r>1.\\
\sum\limits_{j=1}^{t+1} \ \sum\limits_{s=0}^{r-1} ( ^r\mathscr C ^{j-1}+ \ ^{r-1}\mathscr C ^{j-1}) \  ^s\mathscr C ^{t+1-j}-1, & r=1.\end{array}\right. \eqno{(5)}$$

Summing over $r$ we obtain
$$\mathscr C _{t+1}(1)=\sum\limits_{r=0}^t\ ^r\mathscr C ^{t+1}=\ ^0\mathscr C ^{t+1}+\sum\limits_{r=1}^t\ ^r\mathscr C ^{t+1}.$$
Since $\ ^0\mathscr C ^{t+1}=1$ and $^r\mathscr C ^{t+1}$ is given by $(5)$, one has

$$\mathscr C _{t+1}(1)=\sum_{r=1}^t \sum_{j=1}^{t+1} \sum_{s=0}^{r-1}( ^r\mathscr C ^{j-1}+ \ ^{r-1}\mathscr C ^{j-1}) \  ^s\mathscr C ^{t+1-j}. \eqno{(6)}$$

The double sum of $r,s$ requires that $r>s$.  On the other hand if we make the substitution $j \mapsto t+2-j$ in the second summand of the right hand side of $(6)$ and interchange the dummy indices $r,s$, then we recover the terms for which $r\leq s$.  This then gives
$$\mathscr C _{t+1}(1)=\sum\limits_{r=0}^t\sum\limits_{s=0}^t\sum\limits_{j=1}^{t+1}\ ^r\mathscr C ^{j-1}\  ^s\mathscr C ^{t+1-j}= \sum_{j=1}^{t+1} \mathscr C _{j-1}(1)\mathscr C _{t-j+1}(1). \eqno{(7)}$$

Finally since $\mathscr C _0(1)=1$, the first assertion of the lemma results by comparison of $(7)$ with the recursive relation for the $\mathcal C _{t+1},\,t \in \mathbb N$ given in \ref {1.1}.

For the second part of the lemma, let us first assume that $j\ne 1,\, t+1$.  Then from  $(3')$ and $(4')$ one checks that in the sum
$\sum\limits_{r=0}^t \ ^r\mathscr C ^{t+1}_j$ every factor $^u\mathscr C ^{j-1} \ ^v\mathscr C ^{t+1-j},\,u, v \in \mathbb N$ occurs exactly once. It hence equals $\mathscr C _{j-1}(1)\mathscr C _{t+1-j}(1)$.  On the other hand the sum itself equals $|H^{t+1}_j|$.

Finally by Lemma \ref{3.3}, we have $$|H_1^{t+1}| = \sum\limits_{r=1}^t\ ^r\mathscr C ^{t+1}_1 = \sum\limits_{r=0}^{t-1}\ ^r\mathscr C ^t=\mathscr C _t(1) = \mathscr C _0(1)\mathscr C _t(1),$$
and similarly for $|H_{t+1}^{t+1}|.$
\end {proof}

\subsection {Catalan Polynomials}\label{4.2}
As we saw in the paragraph above, the Catalan polynomial $\mathscr C _{t+1}$ is a polynomial of degree $t$. The coefficient of $q^r$ in $\mathscr C _{t+1}$ is equal to $1$ for $r=0$ and for $r\ge 1$ it is given by equation \ref{4.1} $(5)$. In particular, $\mathscr C _{t+1}$ is a monic polynomial of degree $t$ with constant coefficient also equal to $1$.

For any $r\in \mathbb N$, set $^{\le r}\mathscr C^{t+1}=\sum\limits_{k=0}^r\ ^k\mathscr C^{t+1}$. By a similar computation as in the proof of theorem \ref{4.1}, we may compute these numbers to be
$$ ^{\le r}\mathscr C^{t+1}=\sum\limits_{k=0}^r\ ^k\mathscr C^{t+1}=\sum\limits_{j=1}^{t+1}\ ^{\le r}\mathscr C^{j-1}\ ^{\le r-1}\mathscr C^{t+1-j}. \eqno{(8)}
$$

Then the number of diagrams of order $t+1$ having height $\le r$ is equal to the number of planted trees having $t+2$ nodes and height $\le r+3$. Indeed, if we denote by $A_{t+2, r+2}$ the latter, by substituting in the recurrence relation in equation (3) of \cite{BKR}, we obtain

$$A_{t+2, r+2}=\sum\limits_{k=1}^{t+1}A_{k, r+2}A_{t+2-k, r+1}. \eqno{(9)}$$
Moreover, as we may see in Table I of \cite{BKR}, for any $t\in \mathbb N$, $A_{t+2, 2}=1=^{\le 0}\mathscr C^{t+1}$. Hence equations $(8)$ and $(9)$ coincide. One may find in oeis.org (sequence A080934) more objects enumerated by the $^{\le r} \mathscr C^{t+1}$. The coefficient of $q^r$ in the Catalan polynomial $\mathcal C_{t+1}$ is equal to the number of Dyck paths of semilength $t+2$ and height $r+2$, as seen in oeis.org A080936.

%
%
%
%
%

\section{The graph of links}\label{5}

\subsection{Labelling of diagrams} \label{5.1}
A labelling of diagrams is an insertion of an integer into each block. In \cite[Section 3.1]{J}, a labelling on diagrams of order $t+1$ was defined. Let $\mathscr D$ be a diagram of order $t+1$ and denote by $B(i, j)$ the block lying in $C_i$ on the $j$th row. This in particular means that $\operatorname{ht}\, C_i\ge j$. We place a label $b(i, j)$ to all blocks $B(i, j)$ of $\mathscr D$, except for the leftmost (resp. rightmost) block $B(i, j)$ when $j$ is even (resp. odd), which is left blank. Labels are defined inductively as follows: first, we set $b(i, 1)=i$. Let $B(i, j+1)$ be a block on the $(j+1)$-row which is not leftmost (resp. rightmost) if $j$ is odd (resp. even) (and hence it is labelled). Then we insert in $B(i, j+1)$ the number $b(k, j)$, where $k$ is such that $C_k$ is the left (resp. right) neighbour of $C_i$ at height $j$, when $j$ is odd (resp. even). Note that the labelling of a diagram $\mathscr D$ is unique; however, we will refer to a labelled diagram as a tableau. A tableau and its equivalence class are denoted by $\mathscr T,\, [\mathscr T]$ respectively.\\

\noindent {\bf Remark.} Note that the label of a block in a diagram may change after adjunction or removal of dominoes (see the example below). Yet by \cite[Lemma 3.1]{J}, a complete tableau is {\it well-numbered}, that is $b(i, j)=i$ if $j$ is odd and $i\ne t+1$ and $b(i, j)=i-1$ if $j$ is even and $i\ne 1$ (for $j$ odd (resp. even) and $i=t+1$ (resp. $i=1$) the block $B(i, j)$ is empty).\\

\noindent {\bf Example.} As we saw in \ref{2.3.1}, the diagrams
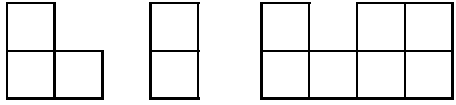
\[\ytableausetup{notabloids}
\begin{ytableau}
*(white) & \none & \none & *(white)\\
*(white) & *(white) & \none & *(white)
\end{ytableau}
\qquad
\ytableausetup{notabloids}
\begin{ytableau}
*(white) & \none & *(white) & *(white)\\
*(white) & *(white) & *(white) & *(white)
\end{ytableau}\]
are equivalent as the second one obtains from the first by adding a left even domino. The corresponding tableaux are
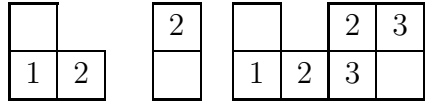
\[\ytableausetup{notabloids}
\begin{ytableau}
*(white) & \none & \none & 2\\
1 & 2 & \none & *(white)
\end{ytableau}
\quad
\ytableausetup{notabloids}
\begin{ytableau}
*(white) & \none & 2 & 3\\
1 & 2 & 3 & *(white)
\end{ytableau}\]
In particular $B(4, 2)$ changes label. Note that the second diagram is complete and, as such, it is well-numbered.

\subsection{Inequalities}\label{5.2}
Let $c_1, c_2, \dots, c_t$ be indeterminates. In \cite[Section 3.2]{J}, a partial order on the $c_i,\, 1\le i\le t$ was assigned to each tableau as follows:
\begin{itemize}
\item Suppose that $j$ is even, then $c_{b(i, j)}\le c_{b(\ell, j-1)}$, with $\ell=k, k+1, \dots, i+1$, where $C_k$ is the left neighbour of $C_i$ at level $j-1$.
\item Suppose that $j$ is odd, then $c_{b(i, j)}\le c_{b(\ell, j-1)}$, with $\ell=i+1, i+2, \dots, k$, where $C_k$ is the right neighbour of $C_i$ at level $j-1$. (Here we have added a zeroth row at the bottom of the tableau with labelling $b(i, 0)=i-1$ for $i=2, 3, \dots, t+1$, with $B(1, 0)$ left blank as the rightmost block of a row at even height).
\end{itemize}

Two tableaux in the same equivalence class might determine distinct partial orders. However, by \cite[Lemma 5.3]{J}, two complete tableaux in the same equivalence class determine the same partial order. \\

\noindent {\bf Definition.} To each equivalence class we assign the partial order of its complete tableaux.\\

\noindent {\bf Example.} The equivalent tableaux below
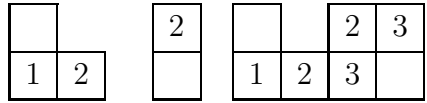
\[\ytableausetup{notabloids}
\begin{ytableau}
*(white) & \none & \none & 2\\
1 & 2 & \none & *(white)
\end{ytableau}
\quad
\ytableausetup{notabloids}
\begin{ytableau}
*(white) & \none & 2 & 3\\
1 & 2 & 3 & *(white)
\end{ytableau}\]
determine the partial orders $c_2\le c_1, c_3$ and $c_2\le c_1$ respectively. We assign to their equivalence class the partial order defined by the complete tableau, namely $c_2\le c_1$.

\subsection{Functions}\label{5.3}
Let $\mathscr T$ be a complete tableau of order $t+1$. In \cite[4.4]{J}, a function $f_\mathscr{T}$ separately linear in the two sets of variables $c_i,\, i \in T$ and $r^j,\, j \in \hat{T}$ was associated with $\mathscr T$ as follows: let $R_u$ be the row of $\mathscr T$ at height $u$. Let $C_{i_1},\, C_{i_2}, \dots, C_{i_k}$ be the columns of $\mathscr T$ of height $\ge u$ and set $$f_{R_u}=\left\{\begin{array}{cc}\sum\limits_{j=1}^{k-1}c_{i_j}(r^{i_j}-r^{i_{j+1}}),& \textrm{if}\quad u \quad \textrm{is odd}.\\
-\sum\limits_{i=2}^kc_{i_j-1}(r^{i_{j-1}}-r^{i_j}),& \textrm{if}\quad u \quad \textrm{is even}.
\end{array}\right.$$
Then set $f_\mathscr{T}=\sum\limits_{u}f_{R_u}$.

By \cite[Lemma 4.5]{J}, $f_\mathscr{T}$ is independent of the choice of $\mathscr T$ in its equivalence class. We therefore write $f_{[\mathscr T]}$ instead of $f_\mathscr{T}$. In \cite [4.7]{J} it was shown that these functions are pairwise distinct, when the $c_i,\, 1\le i\le t$ are viewed as above as indeterminates, or equivalently are ``in general position''. However, we have the following stronger result:

\begin{lemma}
The elements $f_{[\mathscr T]},\, [\mathscr T]\in H^{t+1}$ (as linear functions on the $r^i,\, 1\le i\le t+1$) are pairwise distinct given that the $c_i,\, 1\le i\le t$ take non-zero values and are pairwise distinct.
\end{lemma}
\begin{proof}
This follows from \cite[Section 5.2 (Lemma and $(*)-(***)$)]{J}.
\end{proof}

Recall Section \ref{2.5} and set $f^*_{[\mathscr T]}=f_{[\mathscr T^*]}$. Then $f^*_{[\mathscr T]}$ obtains from $f_{[\mathscr T]}$ by replacing $c_i$ with $c_{t+1-i}$, $r^i$ with $r^{t+2-i}$ and adding $\sum\limits_{i=1}^tc_i(r^i-r^{i+1})$.

\subsection{The graph of links}\label{5.4}
In \cite[3.3.3]{J}, a labelled graph $\mathscr G_{t+1}$ for $H^{t+1}$ was defined as follows. A vertex of $\mathscr G_{t+1}$ is an element $[\mathscr T]$ of $H^{t+1}$. It is given the label $j$ if $[\mathscr T]\in H^{t+1}_j$ (see notation \ref{2.3}). In order to describe the edges we need to recall the notion of a quasi-extremal column of a complete tableau.

A column $C$ of a complete tableau $\mathscr T$ is called left (resp. right) quasi-extremal if it is not strongly extremal but all columns to its left (resp. its right) have height less or equal to $\height\, C$. By \cite[3.3.1]{J}, the height of a quasi-extremal column is $\ge \height\,\mathscr T-1$. Moreover, if it has height $\height\, \mathscr T-1$ it must lie to the left (resp. right) of the strongly extremal column if $\height\,\mathscr T$ is odd (resp. even).  If it has height $\height\, \mathscr T$ it must lie to the right (resp. left) of the strongly extremal column if $\height\,\mathscr T$ is odd (resp. even).

Let $\mathscr T$ be the complete reduced tableau of maximal height in its equivalence class $[\mathscr T]\in H_j^{t+1}$. Then $[\mathscr T]$ and $[\mathscr T']\in H_k^{t+1}$ are linked by an edge of label $i$ if the column $C_k$ of $\mathscr T$ is quasi-extremal and a tableau in $[\mathscr T']$ is obtained by placing a block of label $i$ on top of $C_k$ in $\mathscr T$. (Note that by \cite[Lemma 3.3.2]{J}, placing a box on top of the column $C_k$ of $\mathscr T$ gives a tableau in $H_k^{t+1}$ if and only if $C_k$ is quasi-extremal).

The graph $\mathscr G_{t+1}$ is called the graph of links of $H^{t+1}$. The set of vertices (resp. edges) of $\mathscr G_{t+1}$ is denoted by $V(\mathscr G_{t+1})$ (resp. $E(\mathscr G_{t+1})$). The labels of $V(\mathscr G_{t+1})$ lie in $\hat T:=\{1, 2, \dots, t+1\}$ whereas the labels of $E(\mathscr G_{t+1})$ lie in $T:=\{1, 2, \dots, t\}$. The subset of $V(\mathscr G_{t+1})$ with label $k\in \hat T$ is denoted by $V(\mathscr G_{t+1})^k$. For any graph $\mathscr G$ we will use analogous notation, namely $V(\mathscr G),\, E(\mathscr G),\, V(\mathscr G)^k$ for the sets of vertices, edges and vertices of label $k$ of $\mathscr G$ respectively.

\subsection{Properties of $\mathscr G_{t+1}$}\label{5.5}
We list certain properties of $\mathscr G_{t+1}$, proven in \cite[Section 6]{J}.
\begin{enumerate}[label=(P\arabic*)]
\item No two vertices having the same label are linked by an edge.
\item No two edges having the same label emanate from the same vertex.
\item It is an evaluation graph: we may assign a function $f_v$ to each vertex $v$, such that if $v,\, v'$ with labels $j$ and $k$ respectively are linked by an edge of label $i$, then $f_v-f_{v'}=c_i(r^j-r^k)$. Note that $\mathscr G_{t+1}$ can have cycles, so this is a non-trivial condition.
\item It has a unique chain of labelled vertices and edges
$$\xy \morphism(0,0)|a|/{-}/<400,0>[t+1`t;t] \morphism(400,0)|a|/{-}/<400,0>[t`t-1;t-1]\morphism(800,0)|m|/{--}/<800,0>[t-1`2;] \morphism(1600,0)|a|/{-}/<400,0>[2`1;1]
\endxy$$
We call it the pointed chain.
\item It is connected.
\item It is {\it triadic}: A triad in a labelled graph is a set of vertices $(a, b, c, d)$ successively joined by three edges $(a, b), (b, c), (c, d)$ such that the labels of $(a, b)$ and $(c, d)$ are equal. A graph is called triadic if for every triad the labels of the vertices $a$ and $d$ are equal.
\end{enumerate}

It is clear that all subgraphs of $\mathscr G_{t+1}$ inherit properties (P1), (P2), (P3) and (P6).\\

\noindent {\bf Remarks.}
\begin{enumerate}
\item The function assigned to the vertex of label $t+1$ of the pointed chain of (P4) is called the ``driving function'' and is denoted by $h$. The vertex is denoted by $v_h$. For most of our present purposes, we may and do set $h$ equal to zero. Its dual function $h^*$ is the function corresponding to the vertex of label $1$ in the pointed chain. The function $f_v$ assigned to a vertex $v$ is $f_v=h+f_{[\mathscr T]}$, where $[\mathscr T]$ is the class representing $v\in V(\mathscr G_{t+1})$. See more about the driving function in \ref{5.8.1}.
\item Let $(a, b, c, d)$ be a triad, and let $i$ be the label of $(a, b),\, (c, d)$ and $j$ the label of $(b, c)$. Let $[\mathscr T_a], [\mathscr T_d]$ be the equivalence classes in $H^{t+1}$ corresponding to vertices $a$ and $d$ respectively. By \cite[Lemma 6.6]{J}, the relation $c_i<c_j$ is contained in the partial order assigned to exactly one of $[\mathscr T_a]$ and $[\mathscr T_d]$. We will call this triad an $(i<j)$-triad. We say that an $(i<j)$-triad defines the order relation $c_i<c_j$.
\end{enumerate}

\subsection{$S$-graphs}\label{5.6}
\subsubsection{Subgraphs of $\mathscr G_{t+1}$}\label{5.6.1}
View the $c_i,\,i \in T$ as non-negative integers and let $\textbf{c}$ be the set $\{c_i\}_{i \in T}$ given a linear order. Let $H^{t+1}({\bf c})$ be the subset of $H^{t+1}$ containing the equivalence classes of tableaux whose assigned partial order is compatible with ${\bf c}$. Let $\mathscr G({\bf c})$ be the subgraph of $\mathscr G_{t+1}$ having vertices the elements of $H^{t+1}({\bf c})$. Below we recall the inductive procedure to construct these graphs, given in \cite[Section 7]{J} (the fact that this construction gives the aforementioned subgraphs of $\mathscr G_{t+1}$ is proven in \cite[Theorem 8.5]{J}).

First of all, for $t=1$ the order relations are empty and $\mathscr G_2=\mathscr G({\bf c})$ is the graph
$$\xy \morphism(0,0)|a|/{-}/<400,0>[1`2;1]\endxy.$$

Let ${\bf c}: c_{i_1}< c_{i_2}< \cdots <c_{i_t}$ and set $s:=i_t$ for simplicity.
Set $\textbf{c}'=\textbf{c}\setminus \{c_s\}$ with the following change of labelling. Leave the labels in $[1,s-1]$ unchanged and decrease the labels in $[s+1,t]$ by $1$. Let $\mathscr G(\textbf{c}')$ be the subgraph of $\mathscr G_t$ defined by $\textbf{c}'$ with its induced linear order.

Let $\mathscr G^+$ be the graph obtained from $\mathscr G(\textbf{c}')$ by changing the labelling on edges and on vertices in the following manner. The labels in $[1,s-1]$ are left unchanged, the labels in $[s,t]$ are increased by $1$.  In this graph there are no labels equal to $s$ either on edges or on vertices.

Now define a graph $\mathscr G^-$ isomorphic to $\mathscr G^+$ as an unlabelled graph.  Let $\varphi:\mathscr G^+ \iso \mathscr G^-$ be the resulting unlabelled graph isomorphism.  Then $\varphi$ leaves all labels fixed except that it takes a vertex $v \in V(\mathscr G^+)^{s+1}$ to the vertex $\varphi(v) \in V(\mathscr G^-)^s$. In particular, there are no labels equal to $s$ on edges nor labels equal to $s+1$ on vertices.

Then $\mathscr G(\textbf{c})$ is the union of $\mathscr G^+$ and $\mathscr G^-$ in which each $v \in V(\mathscr G^+)^{s+1}$ is joined by an edge with label $s$ to $\varphi(v) \in V(\mathscr G^-)^s$. In particular, $\mathscr G({\bf c})$ is connected.\\

\noindent {\bf Remark.} Notice that by Remark \ref{5.5} (2), the order defined by the set of all triads in $\mathscr G({\bf c})$ is compatible with the linear order ${\bf c}$.\\

By the inductive construction $\mathscr G({\bf c})$ has $2^t$ vertices. In other words, for all $t!$ possible linear orders ${\bf c}$ there are exactly $2^t$ equivalence classes of tableaux in $H^{t+1}$ whose assigned partial orders are compatible with ${\bf c}$. However, it turns out that not all $\mathscr G({\bf c})$ are distinct even as labelled graphs. In Lemma \ref{6.7} we compute the precise number of distinct labelled graphs to be the Catalan number $\mathcal C_t$.

\subsubsection{Examples} \label{5.6.2} Let us construct the two graphs for $t=2$ from the graph $$\mathscr G:=\mathscr G(c_1): \xy \morphism(0,0)|a|/{-}/<400,0>[1`2;1]\endxy.$$
For $c_1<c_2$ we have $\mathscr G^+: \xy \morphism(0,0)|a|/{-}/<400,0>[1`3;1]\endxy$ and $\mathscr G^-: \xy \morphism(0,0)|a|/{-}/<400,0>[1`2;1]\endxy$. Hence $\mathscr G(c_1<c_2)$ is the graph
$$
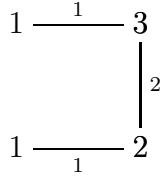
\xy \morphism(0,0)|a|/{-}/<400,0>[1`3;1]\morphism(0,-400)|b|/{-}/<400,0>[1`2;1]\morphism(400,-400)|r|/{-}/<0,400>[2`3;2]\endxy$$

Similarly, for $c_2<c_1$, we have that $\mathscr G^+: \xy \morphism(0,0)|a|/{-}/<400,0>[2`3;2]\endxy$ and $\mathscr G^-: \xy \morphism(0,0)|a|/{-}/<400,0>[1`3;2]\endxy$. Hence $\mathscr G(c_2<c_1)$ is the graph $$
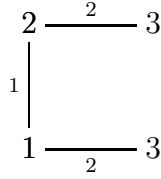
\xy \morphism(0,0)|a|/{-}/<400,0>[2`3;2]
\morphism(0,-400)|b|/{-}/<400,0>[1`3;2]
\morphism(0,-400)|l|/{-}/<0,400>[1`2;1]\endxy$$

Similarly, we construct the graphs for $t=3$ for the linear orders ${\bf c}: c_2<c_1<c_3$ and ${\bf c'}: c_2<c_3<c_1$. We obtain the same graph in both cases, namely the graph :

$$\mathscr G({\bf c})=\mathscr G({\bf c'})=\bfig
\node a(0,0)[2]
\node b(300,300)[4]
\node c(0,-400)[1]
\node d(300,-700)[4]
\node A(1000,0)[2]
\node B(700,300)[3]
\node C(1000,-400)[1]
\node D(700,-700)[3]
\arrow/-/[a`b;2]
\arrow|l|/-/[a`c;1]
\arrow|b|/-/[d`c;2]
\arrow/-/[A`B;2]
\arrow|r|/-/[A`C;1]
\arrow|b|/-/[D`C;2]
\arrow|a|/-/[b`B;3]
\arrow|b|/-/[d`D;3]
\efig$$

For the remaining four linear orders (in the case $t=3$) the graphs are distinct as labelled graphs, but isomorphic to a single tree as unlabelled graphs. This gives altogether 5 distinct labelled graphs.

\subsubsection{}\label{5.6.3} The following Lemma will be useful to us later in Section \ref{6.7}.

\begin{lemma}  The labelled graphs $\mathscr G(c_{i_1}<c_{i_2}<\ldots < c_{i_{t-2}}< c_{i_{t-1}} < c_{i_t})$ and $\mathscr G(c_{i_1}<c_{i_2}<\ldots<c_{i_{t-2}}< c_{i_t}< c_{i_{t-1}})$ coincide if $|i_{t-1}-i_t|>1$.
\end{lemma}

\begin{proof}  Set $i_{t-1}=s$ and $i_t=r$.  We can assume that $s-2\geq r$, equivalently that $s-1\geq r+1$ both of which will be used below.  We follow the construction of \ref{5.6.1} applied twice in the two different ways. 

Set ${\bf c_1}: c_{i_1}<c_{i_2}<\ldots < c_{i_{t-2}}< c_{s} < c_{r}$ and ${\bf c_2}: c_{i_1}<c_{i_2}<\ldots<c_{i_{t-2}}< c_{r}< c_{s}$. Recall \ref{5.6.1} the relabelling for ${\bf c_1'}:={\bf c_1}\setminus \{c_r\},\, {\bf c_2'}:={\bf c_2}\setminus \{c_s\}$ and the notations $\mathscr G({\bf c_i'}),\, \mathscr G({\bf c_i'})^\pm$, for $i=1, 2$.

Set $\textbf{c}''=\textbf{c}\setminus \{c_r,c_s\}$ with the following change of labelling.  Leave the labels in $[1,r-1]$ unchanged, decrease the labels in $[r+1,s-1]$ by $1$ and decrease the labels in $[s+1,t]$ by $2$.  Then $\mathscr G(\textbf{c}'')$ is the subgraph of $\mathscr G_{t-1}$ defined by $\textbf{c}''$ with its induced linear order.

Now let $\mathscr G^{++}$ be the graph obtained from $\mathscr G(\textbf{c}'')$ by changing the labels on edges and on vertices in the following manner.  The labels in $[1,r-1]$ are left unchanged, the labels in $[r,s-2]$ are increased by $1$ and the labels in $[s-1,t-1]$ are increased by $2$.  In this graph there are no labels indexed by $r$ or by $s$ neither on edges nor on vertices.

Now we define three further graphs $\mathscr G^{-+},\mathscr G^{+-},\mathscr G^{--}$ isomorphic to $\mathscr G^{++}$ as unlabelled graphs. Let

$$\bfig
\node a(0,0)[\mathscr G^{++}]
\node b(600,0)[\mathscr G^{-+}]
\node c(0,-400)[\mathscr G^{+-}]
\node d(600,-400)[\mathscr G^{--}]
\arrow|a|/->/[a`b;\varphi_r]
\arrow|a|/->/[c`d;\varphi_r]
\arrow|l|/->/[a`c;\varphi_s]
\arrow|r|/->/[b`d;\varphi_s]
\efig$$
be the commuting diagram of unlabelled graph isomorphisms which leaves labels on edges and vertices fixed except that $\varphi_r$ (resp. $\varphi_s$) takes a vertex with label $r+1$ (resp. $s+1$) to a vertex with label $r$ (resp. $s$).  Notice that we can also view $\varphi_r$ (resp. $\varphi_s$) as an unlabelled graph isomorphism of the $\mathscr G^{++}\times \mathscr G^{+-}$ onto $\mathscr G^{-+}\times \mathscr G^{--}$ (resp.  $\mathscr G^{++}\times \mathscr G^{-+}$ onto $\mathscr G^{+-}\times \mathscr G^{--}$).

Consider the graph $\mathscr G_2^+$ obtained from the pair $\mathscr G^{++},\mathscr G^{-+}$ by joining $v \in V(\mathscr G^{++})^{r+1}$ to $\varphi_r(v) \in V(\mathscr G^{-+})^r$ by an edge with label $r$.  Then $\mathscr G_2^+$ is just $\mathscr G({\bf c_1'})^+$ by the construction of \ref{5.6.1}. Again the graph $\mathscr G_2^-$ obtained from the pair $\mathscr G^{+-},\mathscr G^{--}$ by joining $v \in V(\mathscr G^{+-})^{r+1}$ to $\varphi_r(v)\in V(\mathscr G^{-+})^r$ by an edge with label $r$ is just the graph $\mathscr G({\bf c_1'})^-$.  Note that the graphs $\mathscr G_2^+$ and $\mathscr G_2^-$ are isomorphic as unlabelled graphs. Then the graph we obtain from the pair $\mathscr G_2^+,\mathscr G_2^-$  by joining
$v \in V(\mathscr G_2^+)^{s+1}$ to $\varphi_s(v)\in V(\mathscr G_2^-)^{s}$ by an edge with label $s$ is just $\mathscr G({\bf c_1})$.

Similarly the graph $\mathscr G_1^+$ obtained from the pair $\mathscr G^{++},\mathscr G^{+-}$ by joining $v \in V(\mathscr G^{++})^{s+1}$ to $\varphi_s(v) \in V(\mathscr G^{-+})^s$ by an edge with label $s$ is just $\mathscr G({\bf c_2'})^+$. Again the graph $\mathscr G_1^-$ obtained from the pair $\mathscr G^{-+},\mathscr G^{--}$ by joining $v \in V(\mathscr G^{-+})^{s+1}$ to $\varphi_s(v) \in V(\mathscr G^{--})^s$ by an edge with label $s$ is just $\mathscr G({\bf c_2'})^-$. Again the graphs $\mathscr G_1^+$ and $\mathscr G_1^-$ are isomorphic as unlabelled graphs and the graph we obtain from the pair $\mathscr G_1^+,\mathscr G_1^-$  by joining
$v \in V(\mathscr G_2^+)^{r+1}$ to $\varphi_r(v)\in V(\mathscr G_2^-)^{r}$ by an edge with label $r$ is just $\mathscr G({\bf c_2})$.

Finally the conclusion of the lemma obtains from the fact that both graphs in its conclusion are obtained from the quadruple $\mathscr G^{++},\mathscr G^{-+},\mathscr G^{+-},\mathscr G^{--}$ by joining vertices with label $r+1$ (resp. $s+1$) to corresponding vertices with label $r$ (resp. $s$) by an edge with label $r$ (resp. $s$).
\end{proof}

\subsubsection{} \label{5.6.4}
Recall what is meant by an $(r<s)$-triad (Remark \ref{5.5} (2)).

\begin{cor} Let $\mathscr G_1=\mathscr G(c_{i_1}<c_{i_2}<\cdots <c_{i_{t-2}}<c_s<c_r)$ and $\mathscr G_2=\mathscr G(c_{i_1}<c_{i_2}<\cdots <c_{i_{t-2}}<c_r<c_s)$ with $|s-r|>1$. Then $\mathscr G_1,\, \mathscr G_2$ have neither $(r<s)$ nor $(s<r)$-triads.
\end{cor}

\begin{proof} By Remark \ref{5.6.1} the triads of $\mathscr G_1,\, \mathscr G_2$ must be compatible with the linear orders $c_{i_1}<c_{i_2}<\cdots <c_{i_{t-2}}<c_s<c_r$ and $c_{i_1}<c_{i_2}<\cdots <c_{i_{t-2}}<c_r<c_s$ respectively.  Thus $\mathscr G_1$ cannot have an $(r<s)$-triad and $\mathscr G_2$ cannot have an $(s<r)$-triad. Yet by Lemma \ref{5.6.3} these graphs coincide. Hence the assertion.
\end{proof}

\subsubsection{} \label{5.6.5}
By \cite[Lemma 7.5]{J}, the subgraphs $\mathscr G({\bf c})$ have the following property:
\begin{enumerate}[label=(P\arabic*)]
\setcounter{enumi}{6}
\item For all vertex $v$ and all $k\in T$, there exists a vertex $v'$ of label $k$ such that there exists an {\it ordered path} from $v$ to $v'$, that is a path
$$v=\xy \morphism(0,0)|a|/{-}/<400,0>[v_1`v_2;i_1] \morphism(400,0)|a|/{-}/<400,0>[v_2`v_3;i_2]\morphism(800,0)|m|/{--}/<800,0>[v_3`v_n;] \morphism(1600,0)|a|/{-}/<400,0>[v_n`v_{n+1};i_n]
\endxy=v'$$
with $c_{i_1}< c_{i_2}< \cdots < c_{i_n}$.
\end{enumerate}

A non-empty graph with properties (P1), (P2), (P3), (P5), (P6) and (P7) is called an $S$-graph of order $t$. Note that an $S$-graph has at least one edge for all $t\ge 1$, since it is non-empty and it satisfies (P7). By \cite[Corollary 7.7]{J}, $\mathscr G({\bf c})$ is an $S$-graph.\\

\noindent {\bf Remark.} The graphs $\mathscr G({\bf c})$ contain the pointed chain of $\mathscr G_{t+1}$ \cite[Lemma 7.3]{J}. This can be deduced by the construction of \ref{5.6.1} as well as by the fact that the elements of $H^{t+1}$ represented by the vertices of the pointed chain have the empty partial order assigned to them, hence they belong to $H^{t+1}({\bf c})$ \cite[Lemma 6.3]{J}. Note also that the vertex $v_h$ (having label $t+1$ in the pointed chain) lies in $\mathscr G^+$.

\subsubsection{} \label{5.6.6}
The following result essentially follows from \cite[Lemma 7.5]{J}.

\begin{lemma}
Given $v\in V(\mathscr G({\bf c}))$ and $k\in \hat T$, there exists a unique $v'\in V(\mathscr G({\bf c}))^k$, for which there exists an ordered path from $v$ to $v'$.
\end{lemma}
\begin{proof}
We may show this by induction. For $t=1$ or $2$ the uniqueness of $v'$ is immediate. Recall the notations and construction of \ref{5.6.1} and let $c_s$ be the unique maximal element in ${\bf c}$. Then $\mathscr G({\bf c})$ is obtained by linking $\mathscr G^+,\, \mathscr G^-$ with edges of label $s$, connecting vertices $v$ in $V(\mathscr G^+)^{s+1}$ with vertices $\varphi(v)\in V(\mathscr G^-)^s$. Recall that $\mathscr G^+$ (resp. $\mathscr G^-$) does not have vertices of label $s$ (resp. $s+1$) and both are isomorphic to $\mathscr G({\bf c}\setminus \{c_s\})$ after relabelling.

Given $v\in V(\mathscr G^+)$ (resp. $V(\mathscr G^-)$) and $j\in T\setminus\{s\}$ (resp. $j\in T\setminus \{s+1\}$), since $c_s$ is the unique maximal element in ${\bf c}$, an ordered path from $v$ to a vertex $v'\in V(\mathscr G({\bf c}))^j$ lies entirely in $\mathscr G^+$ (resp. $\mathscr G^-$) and by the induction hypothesis $v'$ is unique. If now $v\in \mathscr G^+$ (resp. $v\in \mathscr G^-$) and $j=s$ (resp. $j=s+1$), an ordered path from $v$ to $v'$ of label $j$ is the concatenation of an ordered path from $v$ to a vertex $v''$ of label $s+1$ (resp. $s$) lying entirely in $\mathscr G^+$ (resp. $\mathscr G^-$) with the unique edge of label $s$ connecting $v''$ with $\mathscr G^-$ (resp. $\mathscr G^-$).
\end{proof}

\noindent {\bf Remark.} The above lemma is not true for an arbitrary $S$-graph, as we will see in Section \ref{7.4}.

\subsubsection{}\label{5.6.7}
\begin{lemma}
Let
$$v=\xy \morphism(0,0)|a|/{-}/<400,0>[v_1`v_2;i_1] \morphism(400,0)|a|/{-}/<400,0>[v_2`v_3;i_2]\morphism(800,0)|m|/{--}/<800,0>[v_3`v_n;] \morphism(1600,0)|a|/{-}/<400,0>[v_n`v_{n+1};i_n]
\endxy=v'$$
be an ordered path from $v$ to $v'$ in $\mathscr G({\bf c})$ and let $k_i$ be the label of $v_i$, for all $i$, with $1\le i\le n+1$.
\begin{enumerate}
\item The $i_j,\, 1\le j\le n$ are pairwise distinct.
\item The $k_j,\, 1\le j\le n+1$ are pairwise distinct.
\end{enumerate}
\end{lemma}
\begin{proof}
The first part follows by the definition of an ordered path; indeed, one has $c_{i_1}<c_{i_2}<\cdots <c_{i_n}$, hence the $i_j$ are pairwise distinct.

The second part follows by induction using the construction \ref{5.6.1} and the fact that an ordered path in $\mathscr G({\bf c})$ either lies entirely in $\mathscr G^+$ or $\mathscr G^-$, or it is the concatenation of an ordered path in $\mathscr G^+$ or $\mathscr G^-$ with an edge connecting $\mathscr G^+$ and $\mathscr G^-$, as in the proof of Lemma \ref{5.6.6}.
\end{proof}

\subsubsection{}\label{5.6.8}
\begin{lemma}
For any $S$-graph, the ordered path from $v$ to $v'$ is unique.
\end{lemma}
\begin{proof}
This results from the evaluation property (P3) of an $S$-graph. Indeed, let
$$v=\xy \morphism(0,0)|a|/{-}/<400,0>[v'_1`v'_2;j_1] \morphism(400,0)|a|/{-}/<400,0>[v'_2`v'_3;j_2]\morphism(800,0)|m|/{--}/<800,0>[v'_3`v'_m;] \morphism(1600,0)|a|/{-}/<400,0>[v'_m`v'_{m+1};j_m]
\endxy=v'$$
be another ordered path. Let $k_p$ be the label of $v_p$, for all $p$, with $1\le p\le n$ and $\ell_q$ the label of $v'_q$, for all $q$, with $1\le q\le m$. Then, by (P3) one has that $$f_v-f_{v'}=\sum\limits_{p=1}^nc_{i_p}(r^{k_p}-r^{k_{p+1}})=\sum\limits_{q=1}^mc_{j_q}(r^{\ell_q}-r^{\ell_{q+1}}).$$ Moreover, by Lemma \ref{5.6.7}, the labels on the vertices of an ordered path are pairwise distinct. Since the $r^{k_p},\, r^{\ell_q}$ are indeterminates, the latter forces $\{k_p\}_{p=1}^n=\{\ell_q\}_{q=1}^m$ and $n=m$. But $v_1=v_1'$ and so $k_1=\ell_1$. This implies that $i_1=j_1$ and then, by (P2), $v_2=v_2'$. Then by induction $v_i=v_i'$, for all $i$, with $1\le i\le n$.
\end{proof}

\subsection{Sources and Sinks}\label{5.7}
Let $\mathscr G$ be a graph with vertices labelled by $\hat T$. Instead of labelling the edges by $T$, assign a direction to each edge. Define an ordered chain from a vertex $v$ to a vertex $v'$ to be a chain
$$v=v_1\rightarrow v_2\rightarrow \cdots \rightarrow v_n=v'.$$
Assume that $V(\mathscr G)^s\ne \emptyset$, for all $s\in \hat T$. In analogy to (P7) we may ask if it is possible to impose the condition (P7)$^\prime$ that for all $v$ and all $s\in \hat T$, there exists $v'\in V(\mathscr G)^s$ and an ordered chain from $v$ to $v'$.

One easily checks that (P7)$^\prime$ implies that $\mathscr G$ must contain cycles. Condition (P7) is subtly different to (P7)$^\prime$. On the other hand, if we fix $s\in \hat T$, say $s=t+1$, we may view an ordered path (in the sense of (P7)) from any $v\in \mathscr G$ to some element $v'\in V(\mathscr G)^{t+1}$ as defining an ordered chain (in the sense of (P7)$^\prime$). Indeed, let
$$v=\xy \morphism(0,0)|a|/{-}/<400,0>[v_1`v_2;] \morphism(400,0)|a|/{-}/<400,0>[v_2`v_3;]\morphism(800,0)|m|/{--}/<800,0>[v_3`v_n;] \morphism(1600,0)|a|/{-}/<400,0>[v_n`v_{n+1};]
\endxy=v'$$
be an ordered path and suppose that for a second element $w$, the path
$$w=\xy \morphism(0,0)|a|/{-}/<400,0>[w_1`w_2;] \morphism(400,0)|a|/{-}/<400,0>[w_2`w_3;]\morphism(800,0)|m|/{--}/<800,0>[w_3`w_m;] \morphism(1600,0)|a|/{-}/<400,0>[w_m`w_{m+1};]
\endxy=w'$$
from $w$ to $w'\in \mathscr G^{t+1}$ meets the path from $v$ to $v'$ at some vertex $v_i=w_j$. Then by Lemma \ref{5.6.6} $v'=w'$. This means that for all $v$, for the unique ordered path from $v$ to $v'\in V(\mathscr G)^{t+1}$ we can choose arrows on edges so that $v_i\rightarrow v_{i+1}$, for all $i=1, 2, \dots, n-1$. Thus the elements of $V(\mathscr G({\bf c}))^{t+1}$ become sinks as well as the roots of the union of trees, where union is $\mathscr G({\bf c})$.

\subsection{Degenerations} \label{5.8}
Recall that ${\bf c}$ is the set $\{c_i\}_{i\in T}$ of non-negative integers endowed with a linear order. Recall \ref{5.3} that the vertices of $\mathscr G_{t+1}$ and $\mathscr G({\bf c})$ represent functions which are linear combinations of the $r^i,\, i \in \hat{T}$. These functions are pairwise distinct if the $c_i$ are in general position (Lemma \ref{5.3}) which we can take to mean that they are non-zero and pairwise distinct.

However, in the construction of the dual Kashiwara functions it happens that some of the $c_i$ are zero or equal. Our aim in this paragraph is to study how the $S$-graphs $\mathscr G({\bf c})$ degenerate if we set $c_i=0$ or $c_i=c_j$, for some $i,\, j\in T$. In particular, we want to show that if two linearly ordered sets ${\bf c}$ and ${\bf c'}$ degenerate (in the above manner) to the same linear order ${\bf c''}$, then the corresponding graphs $\mathscr G({\bf c}),\, \mathscr G({\bf c'})$ (or rather the set of functions each define) degenerate to the same graph (set of functions).

In both these two cases we may proceed inductively, that is in the first case put the smallest element equal to zero and in the second case set nearest neighbours equal to zero. These are rather distinct operations and we shall consider them separately.

\subsubsection{} \label{5.8.1}
Recall Remark \ref{5.5} (1). The driving function takes the form $h=-\sum_{i=1}^tc_im^i$. It has the following significance.  Set $r^i-r^{i+1} = m^i+m^{i+1}$, for all $i \in T$. Then for all $[\mathscr T] \in H^{t+1}_j$ the function $f_{[\mathscr T]} +h$ has zero coefficient of $c_j$. Set $$\mathscr H(\textbf{c})=\{f_v\, |\, v \in \mathscr G(\textbf{c})\} = \{h+f_{[\mathscr T]}\,|\, [\mathscr T] \in H^{t+1}(\textbf{c})\},$$
where the definition of $H^{t+1}({\bf c})$ is given in \ref{5.6.1}.

Fix $c_r,\, 1\le r \leq t$ to be the smallest element of $\textbf{c}$. Our aim is to show that the set of functions $f \in \mathscr H(\textbf{c})$ evaluated at $c_r=0$ can be described in the following manner.

Consider the driving function $h$ evaluated at $c_r=0$.  It can be written as a sum $h=h_1+h_2$, where $h_1= -\sum\limits_{i=1}^{r-1}c_im^i$ and $h_2=-\sum\limits_{i=r+1}^tc_im^i$.

Let ${\bf c_1}$ be the subset $\{c_1,c_2,\ldots,c_{r-1}\}$ of $\textbf{c}$ given its induced order. Then we may construct the $S$-graph $\mathscr G({\bf c_1})$ as a subgraph of $\mathscr G_{r}$ to obtain the functions $\mathscr H({\bf c_1})=\{h_1+f_{[\mathscr T]}\,|\, [\mathscr T]\in H^r({\bf c_1})\}$. Set $\mathscr H_1: = \{h+f_{[\mathscr T]}\,|\, [\mathscr T]\in H^r({\bf c_1})\}$.

Let ${\bf c_2}$ be the subset $\{c_{r+1},c_{r+2},\ldots,c_t\}$ of $\textbf{c}$ given its induced order. Define the graph $\mathscr G_{t-r+1}$ as in \ref{5.4} but by adding $r$ to every label. Construct the $S$-graph $\mathscr G({\bf c_2})$ as a subgraph of $\mathscr G_{t-r+1}$. Then consider the set
$$\mathscr H(\textbf{c})_r:=\{h'+f_{[\mathscr T]}\,|\, [\mathscr T]\in H^{t-r+1}({\bf c_2}),\, h'\in \mathscr H_1\}.$$

\begin{lemma} One has $\{f|_{c_{r}=0}\,|\, f \in\mathscr H(\textbf{c})\}= \mathscr H(\textbf{c})_{r}$.
\end{lemma}

\begin{proof} The proof is by induction on $t$. The assertion is trivial if $t=1$. Let $c_s$ be the unique largest element of $\textbf{c}$.  We can assume that $s\neq r$ for otherwise $t=1$.

Recall the notation and construction of \ref{5.6.1}. Recall Remark \ref{5.6.5} that the vertex $v_h$ belongs to $\mathscr G^+$.   By the induction hypothesis we can assume that the lemma holds for the subgraph $\mathscr G^+$.  On the other hand for all $v \in V(\mathscr G^+)$ we have
$$f_v-f_{\varphi(v)}=c_s(r^{s+1}-r^s). \eqno {(*)}$$

It remains to show that the same formula applies to the elements in $\mathscr H(\textbf{c})_r$ in passing from the subset corresponding to $\mathscr G^+$ to the subset corresponding to $\mathscr G^-$.

Suppose $s<r$. Then $c_s$ is the maximal element in ${\bf c_1}$ and we may apply the construction of \ref{5.6.1} to obtain $\mathscr G({\bf c_1})$ from $\mathscr G({\bf c'_1})$, where $\textbf{c}'_1={\bf c_1}\setminus\{c_s\}$.  In this functions are changed in the manner described by $(*)$.  Similarly if $s>r$, then $c_s$ is the maximal element in ${\bf c_2}$ and we may apply the construction of \ref {5.6.1} to obtain $\mathscr G({\bf c_2})$ from $\mathscr G(\textbf{c}'_2)$, where ${\bf c'_2}={\bf c_2}\setminus\{c_s\}$.  In this functions are again changed in the manner described by $(*)$.

This proves the required assertion and hence the lemma.
\end{proof}

\noindent \textbf{Remarks.}
\begin{enumerate}
\item If several of the coefficients are zero then the lemma means that the functions defined by an $S$-graph can be obtained by applying our construction sequentially to the maximal connected subsets of $T$ in which all coefficients are non-zero (hence positive).  This has some theoretical advantages.  However the main advantage is practical since instead of obtaining $2^{|T|}$ functions in which one sets several $c_i$ equal to zero, one only obtains $2^{|T'|}$ functions where $T'$ is the subset of $T$ indexing the non-zero coefficients.
\item {\it Erasing Edges.} The graph theoretic interpretation of setting the smallest non-zero element $c_r$ of $\textbf{c}$ equal to zero is to erase the edges with label $r$. Of course the resulting graph need not be connected violating (P5) as well as (P4).  This means that we cannot use (P3) to compute the function attached to a given vertex. In the lemma above we made use of the specifics of $\mathscr G(\textbf{c})$ to recover the functions obtained by setting $c_r=0$.  It is not too obvious how to do this for an arbitrary $S$-graph and in particular how to recover (P7). Finally some of the components may be isomorphic leading to unnecessary repetitions.
\end{enumerate}

\subsubsection{}\label{5.8.2}
The second operation when nearest neighbours $c_r<c_s$ in a linear order $\textbf{c}$ are set equal is taken care of by the following Lemma.  In this we can assume by Lemma \ref{5.8.1} that the elements of $\textbf{c}$ are all strictly positive.  Moreover via induction on $t$ we may assume that the nearest neighbours to be set equal are the largest ones.  On the other hand at a previous step further pairs may have been set equal to zero and so we should not assume that the elements of $\textbf{c}$ are all strictly increasing.

\begin{lemma}  Suppose $i_{t-1}=i_t+1$,  Then the labelled graphs $\mathscr G_1:=\mathscr G(c_{i_1}\leq c_{i_2}\leq \ldots \leq c_{i_{t-2}}< c_{i_{t-1}} < c_{i_t})$ and $\mathscr G_2:=\mathscr G(c_{i_1}\leq c_{i_2}\leq \ldots\leq c_{i_{t-2}}< c_{i_t}< c_{i_{t-1}})$ are distinct.  They degenerate to a common graph $\mathscr G_c$ when $c_{i_{t-1}},c_{i_t}$ become equal to a common value $c$.
\end{lemma}

\begin{proof} Set $i_{t-1}=s$ and $i_t=r$.  Then $s=r+1$.  We follow the construction of \ref {5.6.1}.

Set $\textbf{c}''=\textbf{c}\setminus \{c_r,c_s\}$ with the following change of labelling.  Leave the labels in $[1,r-1]$ unchanged and decrease the labels in $[s+1,t]$ by $2$.  Then $\mathscr G:=\mathscr G(\textbf{c}'')$ is the subgraph of $\mathscr G_{t}$ defined by $\textbf{c}''$ with its induced linear order.

Set $\textbf{c}'=\textbf{c}\setminus \{c_s\}$ with the following change of labelling.  Leave the labels in $[1,s-1]$ unchanged and decrease the labels in $[s+1,t]$ by $1$.  Then $\mathscr G':=\mathscr G(\textbf{c}')$ is the subgraph of $\mathscr G_{t}$ defined by $\textbf{c}'$ with its induced linear order. Notice that since $s=r+1$, repeating this construction gives $\mathscr G$.

Now let $\mathscr G^+_r$ (resp. $\mathscr G^+_s$) be the graph obtained from $\mathscr G'$ by changing the labels on edges and on vertices in the following manner. The labels on $[1,r]$ (resp. $[1,r-1]$) are left unchanged and the labels on $[r+1,t]$ (resp. $[r,t]$) are increased by $1$.  In $\mathscr G^+_r$ (resp. $\mathscr G^+_s$) there are no vertices nor edges labelled by $r+1=s$ (resp. $r$).  Moreover there is unlabelled graph isomorphism $\varphi_r:\mathscr G' \iso \mathscr G^+_r$ (resp. $\varphi_s:\mathscr G' \iso \mathscr G^+_s$).

Now define $\mathscr G^-_r$ (resp. $\mathscr G^-_s$) isomorphic to $\mathscr G^+_r$ (resp. $\mathscr G^+_s$) as an unlabelled graph.  Let $\varphi_{r,s}:\mathscr G^+_r \iso \mathscr G^-_r$ (resp. $\varphi_{s,r}:\mathscr G^+_s \iso \mathscr G^-_s$) be the resulting unlabelled graph isomorphism.  Then $\varphi_{r,s}$ (resp. $\varphi_{r,s}$) is deemed to leave all labels fixed except to decrease the label on a vertex with label $s+1$ (resp. $r+1$) by $1$.

Let $\mathscr G_{r,s}$ (resp. $\mathscr G_{s,r}$)  be the union of the graphs $\mathscr G_r^+,\mathscr G_r^-$ (resp. $\mathscr G_s^+,\mathscr G_s^-$) in which a vertex $v$ with label $s+1$ (resp. $r+1$) is joined to the vertex $\varphi_{r,s}(v)$ (which has label $s$ (resp. $r$)) with an edge having label $s$ (resp. $r$).

One checks that $\mathscr G_{r,s}=\mathscr G_2$ (resp. $\mathscr G_{s,r}=\mathscr G_1$) and that they are distinct as labelled graphs.

Now define $\hat{\mathscr G_1},\hat{\mathscr G_2},\hat{\mathscr G_3}$ isomorphic to $\mathscr G$ as unlabelled graphs. Define the unlabelled graph isomorphisms $\varphi_i:\mathscr G\iso \hat{\mathscr G_i}$, which leave the labels in $[1,r-1]$ on edges and vertices unchanged, increase the labels in $[r,t-2]$ (resp. in $[r+1,t-1]$) on edges (resp. vertices) by two and increase the label $r$ on vertices by $i-1$. In particular, these three graphs do not have edges with label $r, s=r+1$.

Now assume that $c_r,c_s$ coalesce to some common value $c$.  Let $\mathscr G_c$ denote the graph which is the union of $\hat{\mathscr G_1},\hat{\mathscr G_2},\hat{\mathscr G_3}$ in which for each $v \in V(\mathscr G)^r$, the vertices $\varphi_i(v),\, i=1,2,3$ are formed into a triangle with edges labelled by $r$ (or $s$).

Now consider a vertex $v \in V(\mathscr G^+_r)^r$.  Then $\varphi_s(\varphi_r^{-1}(v)) \in V(\mathscr G^+_s)^s$.  Moreover $v$ belongs to the subgraph $\mathscr G$ of $\mathscr G^+_r$. Its image under $\varphi_{r,s}$ is a second copy of $\mathscr G$.  When $c_r,c_s$ coalesce to $c$ we may replace the labels $r$ and $s$ on the edges of $\mathscr G_{r,s}$ by either $r$ or $s$ and identify the above two copies of $\mathscr G$ through $\varphi_{r,s}$.  One checks that the resulting graph is $\mathscr G_c$.

Similarly consider a vertex $v \in V(\mathscr G^+_s)^{s+1}$.  Then $\varphi_r(\varphi_s^{-1}(v)) \in V(\mathscr G^+_r)^{s+1}$. Moreover $v$ belongs to the subgraph $\mathscr G$ of $\mathscr G^+_s$. Its image under $\varphi_{s,r}$ is a second copy of $\mathscr G$.  When $c_r,c_s$ coalesce to $c$ we may replace the labels $r$ and $s$ on the edges of $\mathscr G_{s,r}$ by either $r$ or $s$ and identify the above two copies of $\mathscr G$ through $\varphi_{s,r}$.  One checks that the resulting graph is $\mathscr G_c$.

This completes the proof of the lemma.
\end{proof}

\noindent {\bf Remark.} {\it Collapsing Triads to Triangles.} When the elements of $\textbf{c}$ are all non-zero, then putting a non-equal pair $c_r,\, c_s$ equal to some common coefficient $c$ has the graph theoretic interpretation of collapsing triads to triangles.

To describe this suppose that $c_r<c_s$.  Then every $(r<s)$-triad $(a,b,c,d)$ in the original $S$-graph is collapsed to a triangle in which the vertices $a$ and $d$ (which have common label) are identified and the edges in the resulting triangles are all replaced by either $r$ or $s$.  However this is \textit{not} the end of the story because further vertices and edges must be identified.  Indeed the vertices and edges joined to the two identified vertices lying in the corresponding copies of $\mathscr G$ (as noted in the last part of the proof of Lemma \ref {1.5}) are identified.

It is not clear if this process of collapsing triads to triangles can be applied to an arbitrary $S$-graph (that is to say not some $\mathscr G(\textbf{c})$).  It is not even clear that if $c_r<c_s$ and are adjacent in the linear order that the corresponding $S$-graph will admit an $(r<s)$-triad.

One may remark that when the coefficients all become equal to some $c>0$ then all the $t$-hypercubes degenerate to $t$ simplexes ($t+1$ vertices).  However sometimes it is advantageous to consider a $t$-hypercube as a disjoint union of directed trees rooted at the extremal elements (comprising the set $H^{t+1}_{t+1}$).  Then this disjoint union degenerates to the unbranched tree with $t+1$ vertices.\\

\noindent {\bf Example 1.} Let us see how the graphs $\mathscr G(c_1<c_2<c_3)$ and $\mathscr G(c_1<c_3<c_2)$ of order $3$ degenerate when we set $c_2=c_3$. Here we take $r=2,\, s=3$ and we follow the notations of the proof of the previous lemma. The graph $\mathscr G$ is the graph of order $1$ $\mathscr G:=\mathscr G(c_1): \xy \morphism(0,0)|a|/{-}/<400,0>[1`2;1]\endxy$. Then the three graphs $\hat{\mathscr G}_i,\, i=1, 2, 3$ are the following:
$\hat{\mathscr G}_1: \xy \morphism(0,0)|a|/{-}/<400,0>[1`2;1]\endxy$, $\hat{\mathscr G}_2: \xy \morphism(0,0)|a|/{-}/<400,0>[1`3;1]\endxy$ and finally $\hat{\mathscr G}_3: \xy \morphism(0,0)|a|/{-}/<400,0>[1`4;1]\endxy$. Hence $\mathscr G_c$ is :
$$\bfig
\node a(0,0)[3]
\node b(400, 0)[4]
\node c(200,200)[2]
\node A(-200,-200)[1]
\node B(600,-200)[1]
\node C(200,400)[1]
\arrow|b|/-/[a`b;2]
\arrow|a|/-/[a`c;2]
\arrow|r|/-/[b`c;2]
\arrow|a|/-/[A`a;1]
\arrow|r|/-/[c`C;1]
\arrow|a|/-/[b`B;1]
\efig$$
On the other hand, again following the notations of the previous lemma, $$\mathscr G'=\mathscr G(c_1<c_2): 
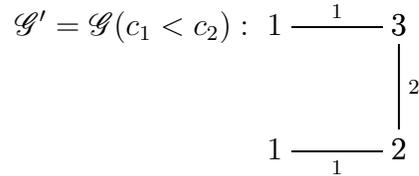
\xy \morphism(0,0)|a|/{-}/<400,0>[1`3;1]\morphism(0,-400)|b|/{-}/<400,0>[1`2;1]\morphism(400,-400)|r|/{-}/<0,400>[2`3;2]\endxy$$
hence the graphs $\mathscr G^+_2, \mathscr G^-_2, \mathscr G^+_3, \mathscr G^-_3$ are respectively:\\
$\mathscr G^+_2:
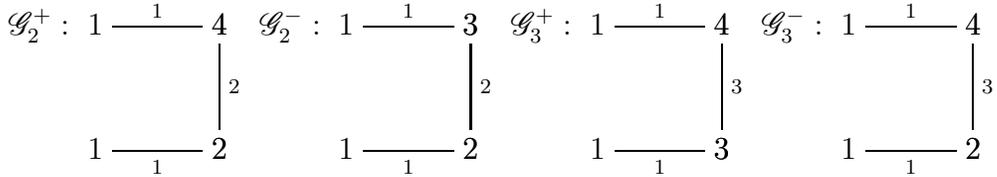
\xy \morphism(0,0)|a|/{-}/<400,0>[1`4;1]\morphism(0,-400)|b|/{-}/<400,0>[1`2;1]\morphism(400,-400)|r|/{-}/<0,400>[2`4;2]\endxy$
$\mathscr G^-_2: \xy \morphism(0,0)|a|/{-}/<400,0>[1`3;1]\morphism(0,-400)|b|/{-}/<400,0>[1`2;1]\morphism(400,-400)|r|/{-}/<0,400>[2`3;2]\endxy$
$\mathscr G^+_3: \xy \morphism(0,0)|a|/{-}/<400,0>[1`4;1]\morphism(0,-400)|b|/{-}/<400,0>[1`3;1]\morphism(400,-400)|r|/{-}/<0,400>[3`4;3]\endxy$
$\mathscr G^-_3: \xy \morphism(0,0)|a|/{-}/<400,0>[1`4;1]\morphism(0,-400)|b|/{-}/<400,0>[1`2;1]\morphism(400,-400)|r|/{-}/<0,400>[2`4;3]\endxy$

Then the graph $\mathscr G_{2, 3}$ obtains by joining the vertex with label $4$ of the first graph with the vertex with label $3$ of the second graph with an edge of label $3$, i.e. it is the graph :

$$\bfig
\node a(0,0)[1]
\node b(400, 0)[4]
\node c(400,-400)[2]
\node d(0, -400)[1]
\node A(1200,0)[1]
\node B(800,0)[3]
\node C(800,-400)[2]
\node D(1200,-400)[1]
\arrow|a|/-/[a`b;1]
\arrow|b|/-/[d`c;1]
\arrow|l|/-/[b`c;2]
\arrow|a|/-/[A`B;1]
\arrow|b|/-/[D`C;1]
\arrow|l|/-/[C`B;2]
\arrow|a|/-/[b`B;3]
\efig$$

Similarly, $\mathscr G_{3, 2}$ obtains from $\mathscr G^+_3, \mathscr G^-_3$ as follows:
$$\bfig
\node a(0,0)[1]
\node b(400, 0)[4]
\node c(400,-400)[3]
\node d(0, -400)[1]
\node A(1200,0)[1]
\node B(800,0)[4]
\node C(800,-400)[2]
\node D(1200,-400)[1]
\arrow|a|/-/[a`b;1]
\arrow|b|/-/[d`c;1]
\arrow|l|/-/[b`c;3]
\arrow|a|/-/[A`B;1]
\arrow|b|/-/[D`C;1]
\arrow|l|/-/[C`B;3]
\arrow|a|/-/[c`C;2]
\efig$$

Then, the two graphs $\mathscr G_{2, 3}, \mathscr G_{3, 2}$ above degenerate to $\mathscr G_c$ when we set $c_2=c_3=c$; the $(2<3)$-triad (resp. $(3<2)$-triad) of $\mathscr G_{2, 3}$ (resp. $\mathscr G_{3, 2}$) degenerates to a triangle by identifying the two vertices with label $2$ (resp. $4$) and by identifying the edges of label $1$ that emanate from the identified vertices.

If we further set $c_1=c=c_2=c_3$, then the graph $\mathscr G_c$ becomes the $3$-simplex :

$$\bfig
\node a(0,0)[3]
\node b(600, 0)[4]
\node c(300,300)[2]
\node A(300,600)[1]
\arrow|b|/-/[a`b;]
\arrow|b|/-/[a`c;]
\arrow|l|/-/[b`c;]
\arrow|l|/-/[A`a;]
\arrow|r|/-/[c`A;]
\arrow|r|/-/[b`A;]
\efig$$

By contrast, recall that the graphs $\mathscr G_1:=\mathscr G(c_2<c_3<c_1)$ and $\mathscr G_2:=\mathscr G(c_2<c_1<c_3)$ are equal and given by the octagon of Section \ref{5.6.2}. By setting $c_2=c_3=c$ the octagon degenerates to the following graph (although for $\mathscr G_2$ this would force $c_1=c$):
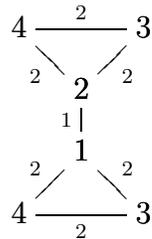
$$\bfig
\node a(0,0)[4]
\node b(400, 0)[3]
\node c(200,200)[1]
\node A(0,600)[4]
\node B(400,600)[3]
\node C(200,400)[2]
\arrow|b|/-/[a`b;2]
\arrow|a|/-/[a`c;2]
\arrow|r|/-/[b`c;2]
\arrow|a|/-/[A`B;2]
\arrow|l|/-/[A`C;2]
\arrow|r|/-/[C`B;2]
\arrow|a|/-/[C`c;1]
\efig$$
If we further set $c_1=c$ in the above graph, then we again obtain the $3$-simplex.\\

\noindent {\bf Example 2.} Let us see how the graphs of order 4, $\mathscr G_1:=\mathscr G(c_1<c_4<c_2<c_3)$ and $\mathscr G_2:=\mathscr G(c_1<c_4<c_3<c_2)$, degenerate when we set $c_2=c_3=c$. With the notation of the lemma, we have $r=2,\, s=3,\, \mathscr G=\mathscr G(c_1<c_2): \xy \morphism(0,0)|a|/{-}/<400,0>[1`3;1]\morphism(0,-400)|b|/{-}/<400,0>[1`2;1]\morphism(400,-400)|r|/{-}/<0,400>[2`3;2]\endxy$. Then we have
$\hat{\mathscr G}_1:
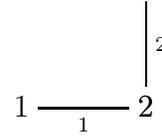
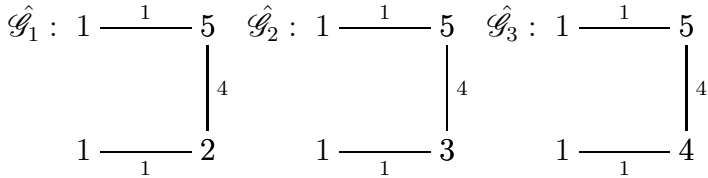
\xy \morphism(0,0)|a|/{-}/<400,0>[1`5;1]\morphism(0,-400)|b|/{-}/<400,0>[1`2;1]\morphism(400,-400)|r|/{-}/<0,400>[2`5;4]\endxy$
$\hat{\mathscr G}_2: \xy \morphism(0,0)|a|/{-}/<400,0>[1`5;1]\morphism(0,-400)|b|/{-}/<400,0>[1`3;1]\morphism(400,-400)|r|/{-}/<0,400>[3`5;4]\endxy$
$\hat{\mathscr G}_3: \xy \morphism(0,0)|a|/{-}/<400,0>[1`5;1]\morphism(0,-400)|b|/{-}/<400,0>[1`4;1]\morphism(400,-400)|r|/{-}/<0,400>[4`5;4]\endxy$

Hence, the graph $\mathscr G_c$ is as follows:
$$\bfig
\node a(0,0)[3]
\node b(400, 0)[4]
\node c(200,200)[2]
\node A(-200,-200)[1]
\node A'(-200,0)[5]
\node A''(-400,-200)[1]
\node B(600,-200)[1]
\node B'(600,0)[5]
\node B''(800,-200)[1]
\node C(400,400)[1]
\node C'(0,400)[5]
\node C''(200,600)[1]
\arrow|b|/-/[a`b;2]
\arrow|a|/-/[a`c;2]
\arrow|r|/-/[b`c;2]
\arrow|r|/-/[A`a;1]
\arrow|r|/-/[c`C;1]
\arrow|l|/-/[b`B;1]
\arrow|a|/-/[A'`a;4]
\arrow|r|/-/[c`C';4]
\arrow|a|/-/[b`B';4]
\arrow|a|/-/[A'`A'';1]
\arrow|l|/-/[C'`C'';1]
\arrow|a|/-/[B'`B'';1]
\efig$$

The graphs $\mathscr G_1,\, \mathscr G_2$ are given below; they degenerate to $\mathscr G_c$ by identifying the two vertices with label 2 (resp. 4) in the $(2<3)$ (resp. $(3<2)$)-triad and identifying the edges emanating from the identified vertices:

$$\mathscr G_1: \bfig
\node a(0,0)[1]
\node b(400, 0)[1]
\node c(800,0)[1]
\node d(1200,0)[1]
\node A(0,400)[5]
\node B(400, 400)[4]
\node C(800,400)[3]
\node D(1200,400)[5]
\node A'(0,800)[5]
\node B'(400, 800)[2]
\node C'(800,800)[2]
\node D'(1200,800)[5]
\node A''(0,1200)[1]
\node B''(400, 1200)[1]
\node C''(800,1200)[1]
\node D''(1200,1200)[1]
\arrow|l|/-/[A`a;1]
\arrow|l|/-/[b`B;1]
\arrow|l|/-/[c`C;1]
\arrow|l|/-/[d`D;1]
\arrow|b|/-/[A`B;4]
\arrow|b|/-/[B`C;3]
\arrow|b|/-/[C`D;4]
\arrow|l|/-/[B`B';2]
\arrow|r|/-/[C`C';2]
\arrow|a|/-/[A'`B';4]
\arrow|a|/-/[C'`D';4]
\arrow|l|/-/[A'`A'';1]
\arrow|l|/-/[B'`B'';1]
\arrow|l|/-/[C'`C'';1]
\arrow|l|/-/[D'`D'';1]
\efig
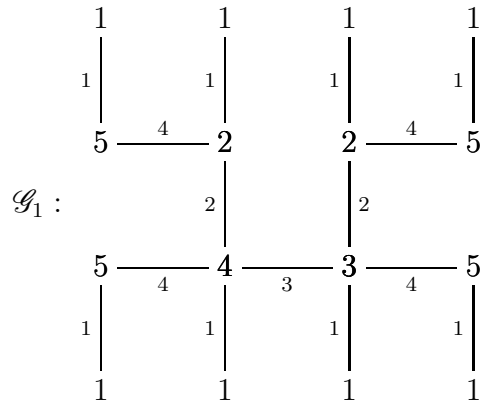
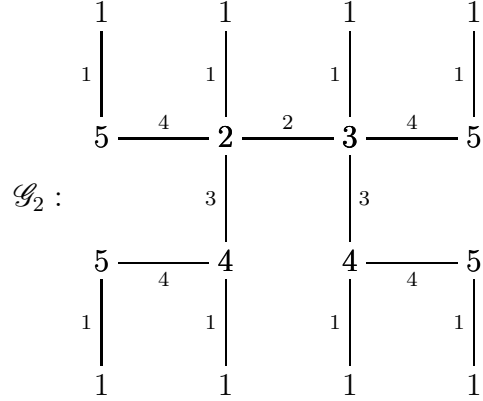
$$

$$\mathscr G_2: \bfig
\node a(0,0)[1]
\node b(400, 0)[1]
\node c(800,0)[1]
\node d(1200,0)[1]
\node A(0,400)[5]
\node B(400, 400)[4]
\node C(800,400)[4]
\node D(1200,400)[5]
\node A'(0,800)[5]
\node B'(400, 800)[2]
\node C'(800,800)[3]
\node D'(1200,800)[5]
\node A''(0,1200)[1]
\node B''(400, 1200)[1]
\node C''(800,1200)[1]
\node D''(1200,1200)[1]
\arrow|l|/-/[A`a;1]
\arrow|l|/-/[b`B;1]
\arrow|l|/-/[c`C;1]
\arrow|l|/-/[d`D;1]
\arrow|b|/-/[A`B;4]
\arrow|b|/-/[C`D;4]
\arrow|l|/-/[B`B';3]
\arrow|r|/-/[C`C';3]
\arrow|a|/-/[A'`B';4]
\arrow|a|/-/[B'`C';2]
\arrow|a|/-/[C'`D';4]
\arrow|l|/-/[A'`A'';1]
\arrow|l|/-/[B'`B'';1]
\arrow|l|/-/[C'`C'';1]
\arrow|l|/-/[D'`D'';1]
\efig$$

\subsubsection{Separation.} \label{5.8.3}
One of the properties of the $S$-graph $\mathscr G(\textbf{c})$ was the functions associated to distinct vertices are distinct. This was called property (P9) and was established in \cite[Thm. 7.8(iv)]{J}.
\begin{lemma} Property (P9) holds for the common graph $\mathscr G_c$ in the conclusion of Lemma \ref{5.8.2}.
\end{lemma}
\begin{proof}
A similar argument to that given in \cite[Thm. 7.8(iv)]{J} applies.  It proceeds by induction on $t$.

We can assume that (P9) holds for $\mathscr G$ defined in the proof of Lemma \ref{5.8.2} by the induction hypothesis. On the other hand $V(\mathscr G_c)=\bigsqcup\limits_{i=1}^3 V(\hat{\mathscr G_i})$.  Now take $v_i\in V(\hat{\mathscr G_i})$ and set $v_j:=\varphi_j\circ \varphi_i^{-1}(v_i)\in \hat{\mathscr G_j}$, where $i, j\in \{1, 2, 3\}$.  Then $f_{v_i}-f_{v_j}= c(r^{i_t+i-1}-r^{i_t+j-1})$, by (P3).  Since only the $c_{i_j},\, 1\leq j\leq t-2$ occur as differences $f_v-f_{v'},\, v,v' \in \hat{\mathscr G_k}$, for any $k \in \{1,2,3\}$ and these coefficients are strictly less than $c$, the assertion follows.
\end{proof}
\smallskip

\noindent {\bf Remark.} Strictly speaking both in \cite [Thm. 7.8(iv)]{J} and in the present analysis we are only proving separation of functions with respect to the coefficients being in general position relative to the partial order imposed on $\textbf{c}$. However in the contribution $(c_u-c_v)(r^u-r^v)$ to $f_{[\mathscr T]}$ coming from \cite[5.2$(*)$]{J}, the indices $u,v$ are not again repeated \cite[Lemma 5.2]{J}.  Thus as far as separation is concerned, the non-vanishing of $(c_u-c_v)$ is equivalent to their being in relative general position, that is to say in this difference taking an indeterminate value.

\section{Hypercubes}\label{6}
\subsection{}\label{6.1}
Fix ${\bf c}$ and recalling (P3), let $f_v$ be the function attached to the vertex $v$ of $\mathscr G({\bf c})$.  Set $f^i=(r^i-r^{i+1})$.  Define an order relation on the set $\{f_v,\,|\, v\in V\}$ 
by \[f \geq f', \quad \textrm{if}\quad f-f' \in \sum_i\mathbb N f^i.\]  Recall Remark \ref{5.5}; the driving function $h$ is the unique minimal element and its dual function $h^*$ is the unique maximal element in this set.

Let $j,\, k$ be the label assigned to the vertices $v,\, v' \in V(\mathscr G({\bf c}))$ respectively; if $v,\, v'$ are joined by an edge of label $i$, we adjoin an arrow on this edge pointing from $v'$ to $v$ if $j<k$. In this case
$$f_v=f_{v'}+c_i\sum_{s=j}^{k-1}f^s,$$
and so $f_v \geq f_{v'}$.

We show that $\mathscr G(\textbf{c})$ is a labelled hypercube in a $t$-dimensional Euclidean space $\mathbb E^t$ with $2^t$ vertices, where exactly the edges joining vertices with the same label are missing and where the co-ordinates of $v \in V$ exceed those $v' \in V$ if $f_v \geq f_{v'}$.  Moreover  the vertices $v_h,\, v_{h^*}$ corresponding to $h,\, h^*$ are interchanged by hypercube inversion, equivalently the unique straight line joining them passes through the origin.  We remark that the latter elegantly interprets the fact that $v_h,v_{h^*}$ lie at the opposite ends of the unique pointed chain having $t+1$ vertices and $t$ edges (Remark \ref{5.5}), whilst $h$ (resp. $h^*$) is the unique minimal (resp. maximal) element of  $\{f_v\,|\, v \in V\}$.

The proof is by induction on $t$.  When $t=1$, then $\mathscr G(\textbf{c})$ is reduced to the pointed chain consisting of just two vertices.  Consequently this case is immediate.

\subsection {}\label{6.2}
Let $c_s$ be the unique largest element of $\textbf{c}$ and take $\textbf{c}':=\textbf{c}\setminus\{c_s\}$ with the induced linear order. Set $\mathscr G:=\mathscr G(\textbf{c}')$ and retain the notations of \ref{5.6.1}. Recall also Remark \ref{5.6.5}, that $v_h$ lies in $\mathscr G^+$.

By the induction hypothesis $\mathscr G^+$ is a hypercube in $\mathbb E^{t-1}$ where $v_h$ can be taken to lie in the bottom left hand corner.  Then its image $\mathscr G^-=\varphi(\mathscr G^+)$ is a hypercube in $\mathbb E^{t-1}$ with $\varphi(v_{h^*})$ lying in the top right hand corner, again by the induction hypothesis.
The additional co-ordinate in the one higher dimensional Euclidean space is used to form a hypercube in $\mathbb E^t$ obtained by joining $v,\varphi(v)$, for all $v \in \mathscr G^+$ having label $s+1$ by an edge with label $c_s$.

By the induction hypothesis, $\mathscr G^\pm$ are hypercubes in $\mathbb E^{t-1}$ with edges omitted exactly if they join vertices having the same index.  Moreover the co-ordinates in $\mathscr G^\pm$ of joined edges $v,v'$ increase if $i_v<i_{v'}$.  It remains to note that by the above construction, this property holds with respect to the edges joined by the new co-ordinate.  Moreover it is clear that the vertices corresponding to the ends of the unique pointed chain are just $v_h$ and $\varphi(v_{h^*})$ respectively and are hence interchanged by hypercube inversion.

\subsection {}\label{6.3}
Fix $\textbf{c},\, c_s$ and ${\bf c'}$ as above.  Set $V_t=V(\mathscr G(\textbf{c})), V_{t-1}=V(\mathscr G(\textbf{c}'))$ and let $V_t^i$ (resp. $V_t^i$) denote the set of vertices of $V_t$ (resp. $V_{t-1}$) of label $i$. We have the following inductive formula for their cardinalities, namely
$$|V_t^i|= \left\{\begin{array}{ll}|V^i_{t-1}|,& i\in \{s,s+1\}, \\
   2|V^i_{t-1}|,&\text{otherwise}. \\

\end{array}\right.\eqno {(*)}$$

In addition the number of edges of $\mathscr G(\textbf{c})$  with label $c_s$ equals $|V_t^s|$. This gives the following result.

\begin{lemma}
\
\begin{enumerate}
\item The number of vertices of $\mathscr G(\textbf{c})$ with label $i \in \hat{T}$ is a power of $2$.
\item The number of edges of $\mathscr G(\textbf{c})$ with label $i \in T$ is a power of $2$.
\end{enumerate}
\end{lemma}

\subsection {}\label{6.4}
The powers that occur in the conclusion of Lemma \ref{6.3} can be computed.  The interesting thing is that they are determined by the following canonical sequence.  Let $(k_1,k_2,\ldots,k_t)$ be the unique set of pairwise distinct elements of $T$ chosen so that the sequence $(c_{k_1},c_{k_2},\ldots,c_{k_t})$ is strictly increasing. 
Now for all $i \in T$, set $n(k_i)=k_i-|\{k_j<k_i|j>i\}|$.  Then $(n(k_1),n(k_2),\ldots,n(k_t))$, viewed as an ordered set is the desired canonical sequence. For example, if the $c_i,\, i \in T$ are increasing, that is $k_i=i$, then $n(k_i)=i$.  If the $c_i,\, i\in T$ are decreasing, that is $k_i=t-i+1$, then $n(k_i)=1$, for all $i \in T$. One may also obtain the canonical sequence inductively by removing the rightmost element and decreasing the remaining elements of the set  by $1$ (resp. by $0$) if they are $>$ (resp. $<$) than this rightmost element.

Recall Section \ref{6.3} that the cardinality of $V^k_t$, for all $k \in \hat{T}$ is a power of $2$, which we denote by $p(k)_t$.

Let us show how to compute $\{p(k)_t\}_{k \in \hat{T}}$.  The argument is simply applying \ref{6.3} $(*)$ using induction on $t$. Let $\{p(k)_i\}_{k=1}^{i+1}$ denote the above set defined by the first $i$ terms of the canonical sequence, namely $(n(k_1),n(k_2),\ldots,n(k_i))$.  If $i=1$, then $(p(1)_1, p(2)_1)=(0,0)$ corresponding to the unique labelled hypercube in dimension one which has one vertex with label $1$ and one vertex with label $2$. Then through \ref{6.3} $(*)$, for all $\{i\,|\, 1\le i< t\}$ we obtain
$$p(k)_{i+1}=p(k)_i+1,\, k <n(k_{i+1}),$$
$$p(k)_{i+1}=p(k+1)_{i+1}=p(k)_i,\, k=n(k_{i+1}),$$
$$p(k+1)_{i+1}=p(k)_i+1,\, k>n(k_{i+1}),$$
from which $\{p(k)_t\}_{k \in \hat{T}}$ may be computed inductively.

Now let $E^k_t$ denote the set of edges of $E(\mathscr G(\textbf{c}))$ with label $k \in T$.  Its cardinality is a power of $2$ which we shall denote by $q(k)_t$. Similarly, let $\{q(k)_i\}_{k=1}^{i}$ denote the above set defined by the first $i$ terms of the canonical sequence, namely $(n(k_1),n(k_2),\ldots,n(k_i))$. When $i=1$, this set is $\{0\}$. Then through \ref{6.3} $(*)$, for all $\{i\,|\,1\le i< t\}$, we obtain
$$q(k)_{i+1}=q(k)_i+1,\, k <n(k_{i+1}),$$
$$q(k)_{i+1}=q(k)_i,\, k=n(k_{i+1}),$$
$$q(k)_{i+1}=q(k)_i+1,\, k>n(k_{i+1}).$$

\subsection {}\label{6.5}
One easily checks that the original sequence $(k_1,k_2,\ldots,k_t)$ can be recovered from the canonical sequence $(n(k_1),n(k_2),\ldots,n(k_t))$ by just reversing our previous recipe, that is to say increasing by $1$ (resp. $0$) every number $\geq n(k_{i+1})$ (resp. $<n(k_{i+1})$) of the $i^{th}$ set adjoining $n(k_{i+1})$ to its right.

On the other hand the label multiplicities $\{p(k)_t\}_{k \in \hat{T}}$ and $\{q(k)_t\}_{k \in T}$ given by different canonical sequences can coincide, as shown in the lemma below.

\begin{lemma} Let $(r_1,r_2,\ldots,r_t)$ be a canonical sequence.  If $r_i>r_{i+1}$, for some $i\in \{1, 2, \dots t\}$, the label multiplicities coincide if we replace the successive elements $r_i,\, r_{i+1}$ by $r_{i+1},\, r_i+1$.
\end{lemma}

\begin{proof}
Denote by $\{p(k)_t\}_{k\in \hat T}$ the vertex label multiplicities corresponding to the original sequence $(\dots, r_i, r_{i+1}, \dots)$ and by $\{p'(k)_t\}_{k\in \hat T}$ the labels corresponding to the sequence $(\dots, r_{i+1}, r_i+1, \dots)$. One has
$$p(k)_i=p(k)_{i-1}+1,\, k <r_i,$$
$$p(k)_i=p(k+1)_i=p(k)_{i-1},\, k=r_i,$$
$$p(k+1)_i=p(k)_{i-1}+1,\, k>r_i,$$
and
$$p(k)_{i+1}=p(k)_i+1,\, k <r_{i+1},$$
$$p(k)_{i+1}=p(k+1)_{i+1}=p(k)_i,\, k=r_{i+1},$$
$$p(k+1)_{i+1}=p(k)_i+1,\, k>r_{i+1}.$$
Then, since $r_{i+1}<r_i<r_i+1$:
$$p(k)_{i+1}=p(k)_{i-1}+2,\, k <r_{i+1},$$
$$p(r_{i+1})_{i+1}=p(r_{i+1}+1)_{i+1}=p(r_{i+1})_{i-1}+1,$$
$$p(k+1)_{i+1}=p(k)_{i-1}+2, \, r_{i+1}<k<r_i,$$
$$p(r_i+1)_{i+1}=p(r_i+2)=p(r_i)_{i-1},$$
$$p(k+1)_{i+1}=p(k+1)_{i-1}+2,\, k>r_i+1.$$
On the other hand,
$$p'(k)_i=p'(k)_{i-1}+1,\, k <r_{i+1},$$
$$p'(k)_i=p'(k+1)_i=p'(k)_{i-1},\, k=r_{i+1},$$
$$p'(k+1)_i=p'(k)_{i-1}+1,\, k>r_{i+1},$$
and
$$p'(k)_{i+1}=p'(k)_i+1,\, k <r_i+1,$$
$$p'(k)_{i+1}=p'(k+1)_{i+1}=p'(k)_i,\, k=r_i+1,$$
$$p(k+1)_{i+1}=p'(k)_i+1,\, k>r_i+1.$$
Hence,
$$p'(k)_{i+1}=p'(k)_{i-1}+2,\, k <r_{i+1},$$
$$p'(r_{i+1})_{i+1}=p'(r_{i+1}+1)_{i+1}=p'(r_{i+1})_{i-1}+1,$$
$$p'(k+1)_{i+1}=p'(k)_{i-1}+2, \, r_{i+1}<k<r_i,$$
$$p'(r_i+1)_{i+1}=p'(r_i+2)=p(r_i)_{i-1},$$
$$p'(k+1)_{i+1}=p'(k+1)_{i-1}+2,\, k>r_i+1.$$
Since $p(k)_{i-1}=p'(k)_{i-1}$, for all $k\in \hat T$, we conclude that $p(k)_{i+1}=p'(k)_{i+1}$ for all $k\in \hat T$, hence $p(k)_t=p'(k)_t$, for all $t\in \hat T$.
By a similar calculation we obtain the equality of edge label multiplicities.
\end{proof}

\subsection{}\label{6.6} By the previous lemma we conclude that we may change the canonical sequence to an increasing sequence ${\bf r}=(r_1\le r_2\le \cdots \le r_t)$, with $r_i\le i$, for all $i$, with $1\le i\le t$. We call the equivalence class of ${\bf r}$ all the canonical sequences which are equal to ${\bf r}$ after a finite number of switchings as in Lemma \ref{6.5}.
\begin{lemma}
For two distinct increasing sequences ${\bf r}=(r_1\le r_2\le \dots \le r_t)$ and ${\bf s}=(s_1\le s_2\le \dots \le s_t)$, the corresponding label multiplicities are distinct.
\end{lemma}

\begin{proof}
Let $\{p^r(k)_t\}_{k\in \hat T}$ (resp. $\{p^s(k)_t\}_{k\in \hat T}$) be the vertex label multiplicities corresponding to the increasing sequence ${\bf r}$ (resp. ${\bf s}$).

Since ${\bf r}\ne {\bf s}$, there exists $m$, with $1<m\le t$ such that $r_i=s_i$ for all $i$, with $1\le i\le m-1$ and $r_m\ne s_m$. Without loss of generality, we may assume that $r_m>s_m$. Then $r_i>s_m$, for all $i$, with $i\ge m$. By the formulae \ref{6.4}, for $k=s_m$, one has $p^r(k)_t=p^r(k)_{m-1}+t-m+1$. On the other hand, again for $k=s_m$, we have $p^s(k)_m=p^s(k)_{m-1}$ and $p^s(k)_t\le p^s(k)_{m-1}+t-m$. Since $p^r(k)_{m-1}=p^s(k)_{m-1}$, we conclude that $p^r(k)_t\ne p^s(k)_t$, hence the assertion.
\end{proof}

\begin{cor}
The labelled graphs corresponding to distinct increasing sequences are distinct.
\end{cor}

\subsection{} \label{6.7}
This leads to the following result.

\begin{lemma} The isomorphism classes of labelled graphs $\mathscr G(\textbf{c})$ are parametrized by increasing sequences $1\leq r_1\leq r_2 \leq \ldots \leq r_t \leq t$, with $r_i\leq i$.  Their number is the Catalan number $\mathcal C_t$.
\end{lemma}
\begin{proof}
It remains to show that for a set of vertex and edge label multiplicities, there exists a unique graph $\mathscr G({\bf c})$.

Let ${\bf k}:=(k_1, k_2, \dots, k_t)$ and ${\bf l}:=(l_1, l_2, \dots, l_t)$ be two canonical sequences in the equivalence class of the increasing sequence ${\bf r}=(r_1\le r_2\le \cdots\le r_t)$ and suppose that they correspond to the linear orders ${\bf c}: (c_{i_1}<c_{i_2}<\cdots <c_{i_t})$ and  ${\bf c'}: (c_{j_1}<c_{j_2}<\cdots <c_{j_t})$ respectively. We will show that $\mathscr G({\bf c})=\mathscr G({\bf c'})$.

We may assume that ${\bf k}$ and ${\bf l}$ differ by a switch, as described in Lemma \ref{6.5}. That is there exists $m$, with $1< m< t-1$, such that $l_m=k_{m+1}$ and $l_{m+1}=k_m+1$, where $k_m>k_{m+1}$ and $l_n=k_n$ for all $n\ne m, m+1$.

We will argue by induction on $t$. For $t=2$ only one canonical sequence forms the equivalence class of ${\bf r}$ and for $t=3$ one may verify the statement by hand.

Notice that by definition of a canonical sequence $k_t=i_t$ and $l_t=j_t$. If $m<t-1$, one has that $i_t=j_t$; then the graphs $\mathscr G({\bf c})$ and $\mathscr G({\bf c'})$ are constructed from graphs $\mathscr G,\, \mathscr G'$ respectively by the method described in \ref{5.6.1}. Then the label multiplicities of $\mathscr G,\, \mathscr G'$ are identical, hence by the induction hypothesis $\mathscr G=\mathscr G'$ and so $\mathscr G({\bf c})=\mathscr G({\bf c'})$.

It remains to examine the case $m=t-1$; the canonical sequences are ${\bf k}=(k_1, k_2, \dots, k_{t-1}, k_t)$ and ${\bf l}=(k_1, k_2, \dots, k_t, k_{t-1}+1)$, with $k_{t-1}>k_t$.

Recall the beginning of Section \ref{6.5} how one obtains the original sequences ${\bf c},\, {\bf c'}$ from the canonical sequences ${\bf k},\, {\bf l}$. Then ${\bf c}: (c_{i_1}<c_{i_2}<\cdots c_{i_{t-2}}<c_{k_{t-1}+1}<c_{k_t})$, whereas ${\bf c'}: (c_{i_1}<c_{i_2}<\cdots c_{i_{t-2}}<c_{k_t}<c_{k_{t-1}+1})$. Moreover, $k_{t-1}+1-k_t>1$. The assertion follows by Lemma \ref{5.6.3}.

Finally, the number of such sequences is equal to $\mathcal C_t$, by \cite[Exercise 6.19 (s)]{S}.
\end{proof}

\section{Characterization of the subgraphs $\mathscr G({\bf c})$}\label{7}
Throughout this section fix a linear order ${\bf c}$ and let $c_s$ be the unique maximal element in ${\bf c}$.

Our aim is to show that $\mathscr G(\textbf{c})$ is uniquely determined as a labelled \textit{sub}graph of $\mathscr G_{t+1}$ by $\textbf{c}$. More precisely, we will show that $\mathscr G({\bf c})$ is the only $S$-subgraph of $\mathscr G_{t+1}$ (satisfying (P7) for ${\bf c}$). In particular, an $S$-subgraph of $\mathscr G_{t+1}$ (satisfying (P7) for ${\bf c}$) contains a unique pointed chain and its triads are compatible with ${\bf c}$ (see Remark \ref{5.5} (2)).

The proof of uniqueness is by induction.  For any $S$-graph $\mathscr G$, we will delete the edges in $\mathscr G$ labelled by $s$ and study the connected components of the resulting graph. It will turn out that this procedure is the reverse of the construction of $\mathscr G({\bf c})$ as described in \ref{5.6.1} and the desired result will follow. We first give some generalities on {\it breaking links}.

\subsection {Breaking Links}\label{7.1}

\subsubsection{} \label{7.1.1}
Let $\mathscr G$ be an $S$-subgraph of $\mathscr G_{t+1}$; it is connected by definition (condition (P5)).  Let $\mathscr C_i,\, i=1,2,\ldots,n$ be the connected components of the graph obtained from $\mathscr G$ by deleting all edges with label $s$.  Let $\ell_{v(\mathscr C_i)}$ denote the set of labels of the vertices of $\mathscr C_i$ and $\ell'_{v(\mathscr C_i)}$ denote the set of labels of the vertices on the $\bigsqcup\limits_{j\neq i}\mathscr C_j$ \textit{directly} joined to a vertex of $\mathscr C_i$ by an edge (necessarily of label $s$).

\begin{lemma}
For every connected component $\mathscr C_i$, one has $\ell_{v(\mathscr C_i)}\cup\ell'_{v(\mathscr C_i)}=\hat{T}$.
\end{lemma}

\begin{proof}  This is immediate from condition (P7).  In more detail, let $v$ be a vertex of $\mathscr C_i$ and choose $j \in \hat{T}$.  Then an ordered chain from $v$ to a vertex $v'$ of label $j$ in $\mathscr G$ must either lie entirely in $\mathscr C_i$ in which case $j \in \ell_{v(\mathscr C_i)}$, or since $c_s$ is the unique largest element of $\textbf{c}$, the vertex $v'$ must be directly linked to a vertex in $\mathscr C_i$ by an edge labelled by $s$ in which case $j \in \ell'_{v(\mathscr C_i)}$.
\end{proof}

\subsubsection {}\label{7.1.2}
Retain the above hypotheses and notation and fix a connected component $\mathscr C_i$ (or simply $\mathscr C$). It is immediate from (P3) that the coefficient of $c_s$ in $f_v$ is independent of the choice of a vertex $v$ of $\mathscr C$.

Fix a complete tableau $\mathscr T$ and let $v$ be the vertex corresponding to $[\mathscr T] \in H^{t+1}$.  Let $u$ be the height of $\mathscr T$.

We claim that for any $k \in T$, the coefficient of $c_k$ in $f_v$ takes the form $r^i-r^j$ for some pair $i,j \in \hat{T}$. Moreover if the coefficient is $r^i-r^j$ with $i<j$, then $i=k$. This follows from \cite[Lemma 5.2]{J}, but we give the details for completeness.

By \ref{5.2}, the label $b(i, j)$ of the block of $\mathscr T$ is $k$ if and only if $i=k$, $j$ is odd and $C_k$ is not right extremal, or $i=k+1$, $j$ is even and $C_{k+1}$ is not left extremal.

Recall the definition of $f_{[\mathscr T]}$ in Section \ref{5.3} and the behaviour of the height function of a complete tableau \cite[Lemma 2.3.3]{J}. Set $a_k$ to be the coefficient of $c_k$ in $f_{[\mathscr T]}$.
\begin{enumerate}
\item Let $\height\, C_{k+1}<\height\, C_k=:u$.
\begin{enumerate}
\item If $u$ is odd, then $\height\, C_{k+1}=u-1$ and $a_k=r^k-r^j$, where $C_j$ is the right neighbour of $C_k$ at height $u$.
\item If $u$ is even, then $\height\, C_{k+1}=u-1$ or $u-2$ and $a_k=r^k-r^j$, where $C_j$ is the right neighbour of $C_k$ at level $u-1$. Then also $\height\, C_j=u-1$.
\end{enumerate}
Suppose $i \in [k,j]\setminus \{k,j\}$.  Then the height of $C_i$ is $\leq u-1$.
\item Let $u:=\height\, C_{k+1}>\height\, C_k$.
\begin{enumerate}
\item If $u$ is odd, then $\height\, C_k=u-1$ or $u-2$ and $a_k=-(r^j-r^k)$, where $C_j$ is the left neighbour of $C_{k+1}$ at level $u-1$. Then $\height\, C_j=u-1$.
\item If $u$ is even, then $\height\, C_k=u-1$ and $a_k=-(r^j-r^k)$, where $C_j$ is the left neighbour of $C_{k+1}$ at level $u$.
\end{enumerate}
Suppose $i \in [j,k]\setminus \{j\}$.  Then the height of $C_i$ is $\leq u-1$.
\item Finally, let $u:=\height\, C_k=\height\, C_{k+1}$.
\begin{enumerate}
\item If $u$ is odd, then $a_k=r^k-r^{k+1}$.
\item If $u$ is even, then $a_k=0$.
\end{enumerate}
\end{enumerate}

\begin{lemma} Let $v\in \mathscr C$ and let $f_v$ be the corresponding function.
\begin{enumerate}[label=(\roman*)]
\item If the coefficient of $c_k$ in $f_v$ is equal to $(r^k-r^j)$, for $k<j$, then $[k,j]\cap \ell_{v(\mathscr C)}\subset \{k\}$ and $[k,j]\cap \ell'_{v(\mathscr C)}\subset \{k, j\}$.
\item If the coefficient of $c_k$ in $f_v$ is equal to $-(r^j-r^k)$, for $k\ge j$, then $[j,k]\cap \ell_{v(\mathscr C)}=\emptyset$ and $[j,k]\cap \ell'_{v(\mathscr C)}\subset \{k\}$.
\end{enumerate}
\end{lemma}

\begin{proof}
For any pair $p,\, q\in \hat T$ with $p<q$, set $[p,q]:=\{i \in \mathbb N\,|\,p\leq i \leq q\}$.

Recall the notion of a strongly extremal column  (Section \ref{2.2}) and the fact that if $C_i$ is the strongly extremal column of $[\mathscr T]$, then the label on the vertex $v$ corresponding to $[\mathscr T]$ is $i$ (Section \ref{5.4}).  

Recall the notion of a quasi-extremal column of a complete tableau $\mathscr T$ and the fact that there is a link of $[\mathscr T]$ to $[\mathscr T']$ with strongly extremal column $C_i$, only if the $i$th column of $\mathscr T$ is quasi-extremal (Section \ref{5.4}). 

For (i) we are in situation (1) or (3) (a) of what precedes the lemma. Then $C_i\in [k,j]$ can only be the strongly extremal column of $\mathscr T$  if $i=k$. Moreover, $C_i$, with $i\in [k, j]$ can be a quasi-extremal column if $i=j$ in case (1) (a) or $i\in \{k, j\}$ in case (1) (b).  This gives (i).

For (ii), we are in situation (2) or (3) (b) of what precedes the lemma. Then $C_i\in [j, k]$ can never be strongly extremal; it can be quasi-extremal only if $i=k$.
\end{proof}

\subsubsection {}\label{7.1.3}
Retain the above notation and conventions. Let $\mathscr C$ be a component. Combining lemmata \ref{7.1.1} and \ref{7.1.2}, one has:

\begin{cor}  Either the coefficient of $c_s$ of $f_v,\, v \in \mathscr C$ is zero, or it equals $r^s-r^{s+1}$.
\end{cor}

\subsection{$S$-graphs in $\mathscr G_{t+1}$}\label{7.2}
\begin{thm}
The graph $\mathscr G({\bf c})$ is uniquely determined as a labelled subgraph of $\mathscr G_{t+1}$ by $\textbf{c}$. In particular, it is the unique $S$-subgraph of $\mathscr G_{t+1}$.
\end{thm}

\begin{proof}
We will prove the statement by induction. For $t=1$ one has $\mathscr G(\textbf{c})=\mathscr G_2$ and the assertion is trivial.

Let $\mathscr G$ be an $S$-subgraph of $\mathscr G_{t+1}$ and delete its edges labelled by $s$.

Let $\mathscr C$ be a component of the graph thus obtained from $\mathscr G$ whose functions have a zero coefficient of $c_s$ and $\mathscr C'$ a component whose functions have a coefficient $r^s-r^{s+1}$ of $c_s$.

Recall Section \ref{2.4} (2) that there is exactly one deplete tableau in each equivalence class in $H^{t+1}$. By \cite [Lemma 4.6]{J}, the set $\mathscr S$ of all deplete tableaux of order $t+1$ with zero coefficient of $c_s$ are exactly those with $C_s$ being an empty column.  By \cite [Lemma 4.7]{J}, $\{[\mathscr T]\,|\,\mathscr T\in \mathscr S\}$ is naturally isomorphic to $H^t$. Hence the subgraph of $\mathscr G_{t+1}$ with set of vertices $\{[\mathscr T]\,|\,\mathscr T\in \mathscr S\}$ is isomorphic to $\mathscr G_t$ and contains $\mathscr C$. Obviously $C_s$ cannot be a strongly extremal column of any $\mathscr T \in \mathscr S$ and a fortiori for any vertex $[\mathscr T]$ of $\mathscr C$, that is to say the label $s$ does not occur on the vertices of $\mathscr C$.

Again the set $\mathscr S'$ of all deplete tableaux of order $t+1$ with zero coefficient of $c_{t-s+1}$ are exactly those with $C_{t-s+1}$ being an empty column.  Its dual $\mathscr S'^*$ consists of all deplete tableaux with $C_{s+1}$ being of height $1$ with a complete bottom row. Then $\{[\mathscr T]\,|\,\mathscr T\in \mathscr S'^*\}$ is naturally isomorphic to $H^t$. The coefficient of $c_s$ in any $f_{[\mathscr T]},\, \mathscr T \in \mathscr S'^*$ is $r^s-r^{s+1}$ and by duality again every function with this property so obtains.  In particular $\mathscr C'$ is a subgraph of the graph with vertices $\{[\mathscr T]\,|\,\mathscr T\in \mathscr S'^*\}$.   Clearly $C_{s+1}$ cannot be a strongly extremal column of any $\mathscr T \in \mathscr S'^*$ and a fortiori of any $[\mathscr T] \in \mathscr C'$, that is to say the label $s+1$ does not occur on the vertices of $\mathscr C'$.

Since the coefficient of $c_s$ in $f_v$,  $v \in \mathscr C$ (resp. $v \in \mathscr C'$) is zero (resp. $r^s-r^{s+1}$), it follows by the definition of a link (Section \ref{5.4}) that an edge with label $s$ in $\mathscr G$ can only join a vertex with label $s+1$ in $\mathscr C$ to a vertex with label $s$ in $\mathscr C'$.  Since $c_s$ is the unique largest element of $\textbf{c}$ it follows (just as in the proof of Lemma \ref {7.2}) that $\mathscr C, \mathscr C'$ must satisfy (P7) for ${\bf c'}$.  Thus by the induction hypothesis they must, under the appropriate relabelling, be just $\mathscr G(\textbf{c}')$.  Then recalling \ref{5.6.1} how $\mathscr G(\textbf{c})$ is obtained from $\mathscr G(\textbf{c}')$ it follows that $\mathscr G$ is $\mathscr G(\textbf{c})$. Hence the assertion.
\end{proof}

\begin{cor}
A proper subgraph of an $S$-graph of order $t$ of $\mathscr G_{t+1}$ (i.e. a subgraph of a $\mathscr G({\bf c})$) is not an $S$-graph of order $t$. In other words, $\mathscr G({\bf c})$ are minimal $S$-graphs of order $t$.
\end{cor}

\subsection{} \label{7.4}
The $\mathscr G({\bf c})$ is not the unique $S$-graph for the order ${\bf c}$ in general, only as a subgraph of $\mathscr G_{t+1}$. Indeed, consider ${\bf c} : c_2<c_3<c_1$. Then as we saw in \ref{5.6.2}, $\mathscr G({\bf c})$ is the graph:
$$\bfig
\node a(0,0)[2]
\node b(300,300)[4]
\node c(0,-400)[1]
\node d(300,-700)[4]
\node A(1000,0)[2]
\node B(700,300)[3]
\node C(1000,-400)[1]
\node D(700,-700)[3]
\arrow/-/[a`b;2]
\arrow|l|/-/[a`c;1]
\arrow|b|/-/[d`c;2]
\arrow/-/[A`B;2]
\arrow|r|/-/[A`C;1]
\arrow|b|/-/[D`C;2]
\arrow|a|/-/[b`B;3]
\arrow|b|/-/[d`D;3]
\efig$$
The graph:
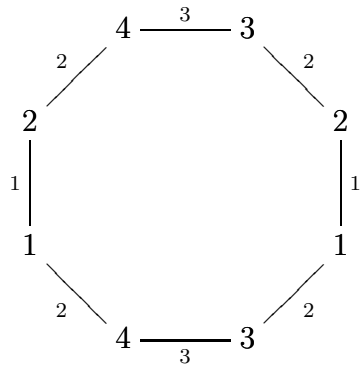
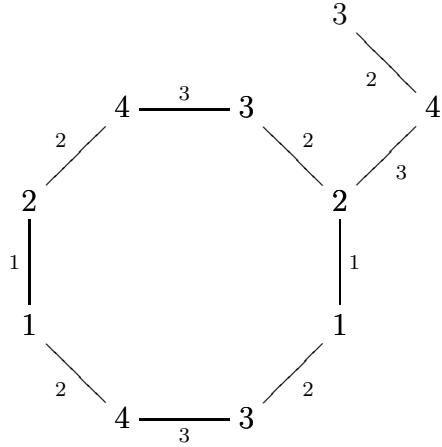
$$\bfig
\node a(0,0)[2]
\node b(300,300)[4]
\node c(0,-400)[1]
\node d(300,-700)[4]
\node A(1000,0)[2]
\node B(700,300)[3]
\node C(1000,-400)[1]
\node D(700,-700)[3]
\node A'(1300,300)[4]
\node A''(1000,600)[3]
\arrow/-/[a`b;2]
\arrow|l|/-/[a`c;1]
\arrow|b|/-/[d`c;2]
\arrow/-/[A`B;2]
\arrow|r|/-/[A`C;1]
\arrow|b|/-/[D`C;2]
\arrow|a|/-/[b`B;3]
\arrow|b|/-/[d`D;3]
\arrow|b|/-/[A`A';3]
\arrow|b|/-/[A'`A'';2]
\efig$$
is an $S$-graph that contains $\mathscr G({\bf c})$.

There are $S$-graphs which do not contain a $\mathscr G({\bf c})$, as the graph below:
$$\bfig
\node a(0,0)[2]
\node b(300,300)[3]
\node c(0,-400)[4]
\node d(300,-700)[1]
\node e(1400,0)[2]
\node f(1800,0)[3]
\node p(2200,0)[4]
\node q(2200,-400)[1]
\node r(2500,300)[2]
\node s(2500,-700)[2]
\node P(3200,0)[4]
\node Q(3200,-400)[1]
\node R(2900,300)[3]
\node S(2900,-700)[3]
\node A(1000,0)[1]
\node B(700,300)[4]
\node C(1000,-400)[3]
\node D(700,-700)[2]
\arrow/-/[a`b;1]
\arrow|l|/-/[a`c;2]
\arrow|b|/-/[d`c;3]
\arrow/-/[A`B;3]
\arrow|r|/-/[A`C;2]
\arrow|b|/-/[D`C;1]
\arrow|a|/-/[b`B;2]
\arrow|b|/-/[d`D;2]
\arrow|a|/-/[A`e;1]
\arrow|a|/-/[e`f;2]
\arrow|a|/-/[p`f;3]
\arrow/-/[p`r;2]
\arrow|l|/-/[p`q;1]
\arrow|b|/-/[s`q;2]
\arrow/-/[P`R;2]
\arrow|r|/-/[P`Q;1]
\arrow|b|/-/[S`Q;2]
\arrow|a|/-/[r`R;3]
\arrow|b|/-/[s`S;3]
\efig$$

In the above graph, let $v$ be the vertex of label $2$ in the chain connecting the two octagons. Note that there are ordered paths from $v$ to two vertices of label $1$, namely $$\xy \morphism(0,0)|a|/{-}/<400,0>[2`1;1]\endxy$$ and $$\xy \morphism(0,0)|a|/{-}/<400,0>[2`3;2] \morphism(400,0)|a|/{-}/<400,0>[3`4;3]\morphism(800,0)|a|/{-}/<400,0>[4`1;1]\endxy.$$
In particular, $v'$ is not uniquely determined by $v$ in (P7) for an arbitrary $S$-graph, unlike the case of $\mathscr G({\bf c})$ (Lemma \ref{5.6.6}).

Moreover, note that the triads of the above $S$-graph are not compatible with ${\bf c}$, unlike the case of $\mathscr G({\bf c})$ (Remark \ref{5.6.1}). Indeed, the above $S$-graph possesses $(1<2)$ as well as $(2<1)$ triads.

\subsection{Tableaux versus Ideals} \label{7.3}

\subsubsection{}
Let us examine the example that we mentioned in Section \ref{1.4} of the introduction, namely compare the graph $\mathscr N_3$ corresponding to the set of nilpotent ideals in the Borel subalgebra of $\mathfrak{sl}_3$ and the graph of links $\mathscr G_3$. Below we draw $\mathscr N_3$ on the left and $\mathscr G_3$ on the right; edges in $\mathscr N_3$ correspond to inclusion of ideals:

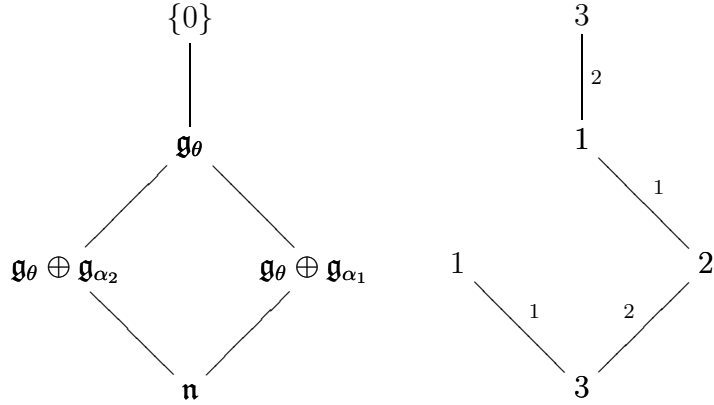
\begin{center}
$\bfig
\node a(800,0)[\mathfrak g_{\theta}\oplus \mathfrak g_{\alpha_1}]
\node b(0,0)[\mathfrak g_{\theta}\oplus \mathfrak g_{\alpha_2}]
\node c(400,400)[\mathfrak g_\theta]
\node d(400,-400)[\mathfrak n]
\node A(400,800)[\{0\}]
\arrow/-/[c`A;]
\arrow/-/[b`c;]
\arrow/-/[b`d;]
\arrow/-/[a`d;]
\arrow/-/[a`c;]
\efig$$
\qquad
$$\bfig
\node a(800,0)[2]
\node b(0,0)[1]
\node c(400,400)[1]
\node d(400,-400)[3]
\node A(400,800)[3]
\arrow/-/[c`A;2]
\arrow/-/[b`d;1]
\arrow/-/[a`d;2]
\arrow/-/[a`c;1]
\efig$
\end{center}

We know of $\mathcal C_2=2$ subsets of the set of nilpotent ideals of cardinality $2^2=4$; namely, the commutative ideals and the nilradicals of the parabolic subalgebras. The subgraph corresponding to the latter is not connected. In contrast, the two $S$-subgraphs with 4 vertices of $\mathscr G_3$ are connected (we constructed them in Section \ref{5.6.2}).

\subsubsection{One more example.}

The graph of links $\mathcal G_4$ is the graph:
$$\bfig
\node a(0,0)[1]
\node b(300,300)[2]
\node c(0,-400)[3]
\node d(300,-700)[4]
\node e(-300, 300)[4]
\node f(0, 600)[4]
\node g(0,1000)[1]
\node A(1000,0)[4]
\node B(700,300)[3]
\node C(1000,-400)[2]
\node D(700,-700)[1]
\node E(1300, 300)[1]
\node F(1000, 600)[1]
\node G(1000,1000)[4]
\arrow/-/[a`b;1]
\arrow|l|/-/[a`c;2]
\arrow|b|/-/[d`c;3]
\arrow/-/[A`B;3]
\arrow|r|/-/[A`C;2]
\arrow|b|/-/[D`C;1]
\arrow|a|/-/[b`B;2]
\arrow|b|/-/[d`D;2]
\arrow|b|/-/[a`e;3]
\arrow|a|/-/[A`E;1]
\arrow|b|/-/[b`f;3]
\arrow|a|/-/[B`F;1]
\arrow|l|/-/[f`g;1]
\arrow|r|/-/[F`G;3]
\efig$$

In contrast, the graph $\mathscr N_4$ of nilpotent ideals in the Borel subalgebra of $\mathfrak{sl}_4$ is :
$$\bfig
\node a(0,0)[\circ]
\node b(300,-300)[\circ]
\node c(600,0)[\circ]
\node d(300,300)[\circ]
\node e(-200, 400)[\circ]
\node f(100, 700)[\circ]
\node g(300,900)[\circ]
\node A(0,200)[\circ]
\node B(300,-100)[\circ]
\node C(600,200)[\circ]
\node D(300,500)[\circ]
\node E(800, 400)[\circ]
\node F(500, 700)[\circ]
\node G(300,1100)[\circ]
\arrow/-/[a`b;]
\arrow/-/[b`c;]
\arrow/-/[d`c;]
\arrow/-/[d`a;]
\arrow/-/[A`B;]
\arrow/-/[B`C;]
\arrow/-/[D`C;]
\arrow/-/[D`A;]
\arrow/-/[a`A;]
\arrow/-/[b`B;]
\arrow/-/[c`C;]
\arrow/-/[d`D;]
\arrow/-/[A`e;]
\arrow/-/[C`E;]
\arrow/-/[F`D;]
\arrow/-/[f`D;]
\arrow/-/[e`f;]
\arrow/-/[F`E;]
\arrow/-/[F`g;]
\arrow/-/[f`g;]
\arrow/-/[g`G;]
\efig$$

It is an easy exercise to show that one cannot obtain the graph $\mathscr N_4$ by just adjoining edges to $\mathscr G_4$, as in the case $t=3$. Indeed, suppose one can obtain $\mathscr G_4$ from $\mathscr N_4$ in this fashion and define the valence of a vertex to be the number of edges that emanate from it. We name the vertices of the two graphs as follows:
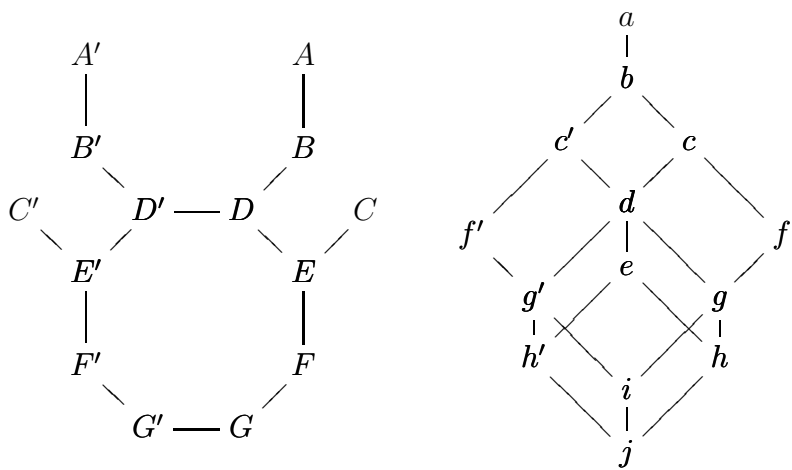
\begin{center}
$\bfig
\node a(0,0)[E']
\node b(200,200)[D']
\node c(0,-300)[F']
\node d(200,-500)[G']
\node e(-200, 200)[C']
\node f(0, 400)[B']
\node g(0,700)[A']
\node A(700,0)[E]
\node B(500,200)[D]
\node C(700,-300)[F]
\node D(500,-500)[G]
\node E(900, 200)[C]
\node F(700, 400)[B]
\node G(700,700)[A]
\arrow/-/[a`b;]
\arrow|l|/-/[a`c;]
\arrow|b|/-/[d`c;]
\arrow/-/[A`B;]
\arrow|r|/-/[A`C;]
\arrow|b|/-/[D`C;]
\arrow|a|/-/[b`B;]
\arrow|b|/-/[d`D;]
\arrow|b|/-/[a`e;]
\arrow|a|/-/[A`E;]
\arrow|b|/-/[b`f;]
\arrow|a|/-/[B`F;]
\arrow|l|/-/[f`g;]
\arrow|r|/-/[F`G;]
\efig$$
\qquad
$$\bfig
\node a(0,0)[h']
\node b(300,-300)[j]
\node c(600,0)[h]
\node d(300,300)[e]
\node e(-200, 400)[f']
\node f(100, 700)[c']
\node g(300,900)[b]
\node A(0,200)[g']
\node B(300,-100)[i]
\node C(600,200)[g]
\node D(300,500)[d]
\node E(800, 400)[f]
\node F(500, 700)[c]
\node G(300,1100)[a]
\arrow/-/[a`b;]
\arrow/-/[b`c;]
\arrow/-/[d`c;]
\arrow/-/[d`a;]
\arrow/-/[A`B;]
\arrow/-/[B`C;]
\arrow/-/[D`C;]
\arrow/-/[D`A;]
\arrow/-/[a`A;]
\arrow/-/[b`B;]
\arrow/-/[c`C;]
\arrow/-/[d`D;]
\arrow/-/[A`e;]
\arrow/-/[C`E;]
\arrow/-/[F`D;]
\arrow/-/[f`D;]
\arrow/-/[e`f;]
\arrow/-/[F`E;]
\arrow/-/[F`g;]
\arrow/-/[f`g;]
\arrow/-/[g`G;]
\efig$
\end{center}
Since by adding edges to $\mathscr G_4$ the valence of its vertices increases or remains the same, vertices of $\mathscr G_4$ that could correspond to the vertex $a$ of $\mathscr N_4$ are the $A, A', C, C'$. By the symmetry of $\mathscr G_4$ it is enough to investigate the cases $A, C$.

If we take $A$ to be the vertex corresponding to $a$, then necessarily $B$ corresponds to $b$ and by the symmetry of $\mathscr N_4$ we may assume that $D$ corresponds to $c$. Both $c$ and $D$ have valence 3. But $c$ is connected to $f$ which has valence 2, whereas $D$ is connected to $D', E$ which have both valence 3. Hence a contradiction.

If now we take $C$ to be the vertex corresponding to $a$, then $E$ corresponds to $b$ and $\{c, c'\}$ to $\{D, F\}$. By the symmetry of $\mathscr N_4$ we may take $D$ to be $c$ and $F$ to be $c'$. This forces $B$ to be $f$, $A$ to be $g$ and $D'$ to be $d$. Since $c'$ and $g$ are linked to $d$ in $\mathscr N_4$, we draw an edge linking $F$ with $D'$ and an edge linking $A$ with $D'$. Then $G$ necessarily corresponds to $f'$ and $G'$ to $g'$. Since $g'$ is linked to $d$, we draw an edge between $D'$ and $G'$. But then the valence of $D'$ becomes 6 which exceeds the valence of $d$, which is 5. Again a contradiction. Thus we conclude that $\mathscr G_4$ does not obtain from $\mathscr N_4$ by just adding edges.

Now consider the directed graph corresponding to $\mathscr G_4$ where arrows point to the vertex obtained by adjoining a block as described in \ref{5.4} :
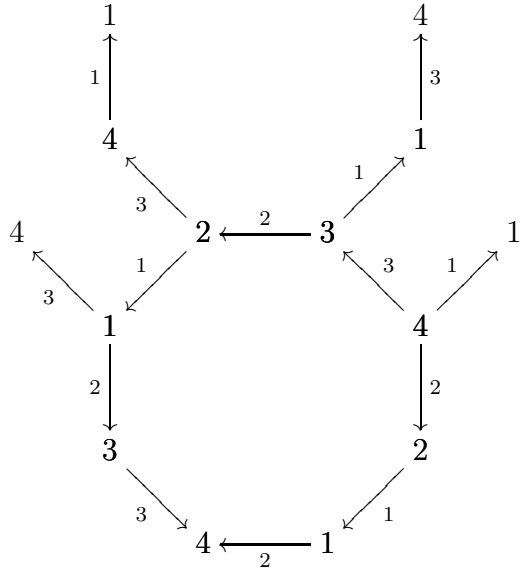
$$\bfig
\node a(0,0)[1]
\node b(300,300)[2]
\node c(0,-400)[3]
\node d(300,-700)[4]
\node e(-300, 300)[4]
\node f(0, 600)[4]
\node g(0,1000)[1]
\node A(1000,0)[4]
\node B(700,300)[3]
\node C(1000,-400)[2]
\node D(700,-700)[1]
\node E(1300, 300)[1]
\node F(1000, 600)[1]
\node G(1000,1000)[4]
\arrow/<-/[a`b;1]
\arrow|l|/->/[a`c;2]
\arrow|b|/<-/[d`c;3]
\arrow/->/[A`B;3]
\arrow|r|/->/[A`C;2]
\arrow|b|/<-/[D`C;1]
\arrow|a|/<-/[b`B;2]
\arrow|b|/<-/[d`D;2]
\arrow|b|/->/[a`e;3]
\arrow|a|/->/[A`E;1]
\arrow|b|/->/[b`f;3]
\arrow|a|/->/[B`F;1]
\arrow|l|/->/[f`g;1]
\arrow|r|/->/[F`G;3]
\efig$$
A maximal chain has length 6. In contrast, a maximal chain in $\mathscr N_4$ (which is a directed graph by inclusion of ideals) has length 7.

\section*{Index of notations}
\noindent Notations are listed below where they first appear.\\\\
1.1 $\mathcal C_t$\\
2.1 $\mathscr D$, $C_i$, $\height\, C,\, \height\, \mathscr D$\\
2.3.1 $\widehat{\mathscr D}$\\
2.3.2 $C^+$, $\mathscr D^+$\\
2.3.4 $[\mathscr D]$, $H^{t+1}$, $H^{t+1}_j$\\
2.5 $\mathscr D^*$\\
3.2 $\height\, [\mathscr D]$, $^sH^r$, $^sH^r_j$, $^{\le s}H^r_j$\\
4.1 $^s\mathscr C^r$, $^s\mathscr C^r_j$, $\mathscr C_{t+1}(q)$\\
5.1 $\mathscr T$, $[\mathscr T]$\\
5.3 $f_{[\mathscr T]}$\\
5.4 $T$, $\hat T$, $\mathscr G_{t+1}$, $V(\mathscr G_{t+1})$, $E(\mathscr G_{t+1})$, $V(\mathscr G_{t+1})^k$\\
5.5 $f_v$, $h$, $v_h$\\
5.6.1 ${\bf c}$, $H^{t+1}({\bf c})$, $\mathscr G({\bf c})$, $\mathscr G^+$, $\mathscr G^-$, $\varphi$\\
6.4 $n(k_i)$, $p(k)_i$, $q(k)_i$\\
7.1.1 $\ell_{v(\mathscr C_i)}$, $\ell'_{v(\mathscr C_i)}$

\section*{Index of notions}
\noindent Notions are listed below where they first appear.\\\\
1.1 Catalan set.\\
2.1 Diagram of order $t+1$, height of diagram.\\
2.2 Left/right extremal column, strongly extremal column, boundary conditions.\\
2.3.1 Domino, Complete/deplete diagram.\\
2.3.4 Equivalence class of diagrams.\\
2.5 Dual diagram.\\
4.1 Catalan polynomial.\\
5.1 Labelling of diagrams, tableau.\\
5.2 Partial order assigned to an equivalence class of tableaux.\\
5.3 Function assigned to an equivalence class of tableaux.\\
5.4 Left/right quasi-extremal column, graph of links.\\
5.5 Pointed chain, $(i<j)$-triad, ordered path, driving function.\\
5.6.5 $S$-graph.\\
6.4 Canonical sequence.

\end {document}